\documentclass[12pt,centertags]{amsart}
\usepackage[latin1]{inputenc} 
\usepackage{amsmath,amstext,amsthm,a4,amssymb,amscd}
\usepackage[mathscr]{eucal}
\usepackage{mathrsfs}
\usepackage{epsf}
\numberwithin{equation}{section}

\newcommand{\field}[1]{\mathbb{#1}}
\newcommand{\Z}{\field{Z}}
\newcommand{\R}{\field{R}}
\newcommand{\C}{\field{C}}

 \def\cC{\mathscr{C}}

\def\mH{\mathcal{H}}
\def\mI{\mathcal{I}}
\def\mJ{\mathcal{J}}

\def\mQ{\mathcal{Q}}

\def\mT{\mathcal{T}}

\newcommand{\db}{\overline\partial}

\DeclareMathOperator{\End}{End}
\DeclareMathOperator{\Hom}{Hom}

\DeclareMathOperator{\Id}{Id}

\DeclareMathOperator{\dive}{div}

\newtheorem{thm}{Theorem}[section]
\newtheorem{lemma}[thm]{Lemma}
\newtheorem{prop}[thm]{Proposition}

\newtheorem{qtn}[thm]{Question}
\newtheorem{conj}[thm]{Conjecture}
\theoremstyle{definition}
\newtheorem{rem}[thm]{Remark}
\newtheorem{exa}[thm]{Example}
\theoremstyle{definition}
\newtheorem{defn}[thm]{Definition}
\newcommand{\be}{\begin{eqnarray}}
\newcommand{\ee}{\end{eqnarray}}
\newcommand{\ov}{\overline}
\newcommand{\wi}{\widetilde}

\newcommand{\comment}[1]{}
\renewcommand{\d}{\partial}

\begin{document}

\title{Characteristic Laplacian in sub-Riemannian geometry}

\date{\today}
\author{Jeremy Daniel and Xiaonan Ma}

\address{Universit{\'e} Paris Diderot - Paris 7,
UFR de Math{\'e}matiques, Case 7012,
75205 Paris Cedex 13, France}
\email{danielj@math.jussieu.fr, ma@math.jussieu.fr}


\begin{abstract} We study a Laplacian operator 
    related to the characteristic cohomology of a smooth manifold 
    endowed with a distribution. We prove that this Laplacian 
    does not behave very well: it is not hypoelliptic in general 
    and does not respect the bigrading on forms in 
    a complex setting. We also discuss the consequences 
    of these negative results for a conjecture of 
    P. Griffiths, concerning the characteristic cohomology 
    of period domains.
\end{abstract}
\maketitle

\setcounter{section}{-1}
\section{Introduction}
Let $X$ be a smooth manifold and denote by $\Omega^\bullet(X)$
its differential graded algebra of smooth differential forms. 
Given a constant-rank distribution $W$ on $X$, we consider the 
\emph{Pfaffian system} associated to $W$, that is the graded 
differential ideal $\mJ^\bullet$ generated by the smooth global 
sections of the annihilator of $W$ in $T^\ast X$. Pfaffian systems 
constitute an important class of \emph{exterior differential systems}, 
for which we refer the interested reader to \cite{EDS}.

For a differential map $f$ from a smooth manifold $Y$ to $X$, 
it is equivalent for the differential of $f$ to have its values 
in the distribution $W$ and for the pullback by $f$ of any form in 
$\mJ$ to vanish on $Y$; we call such maps \emph{solutions} of 
the Pfaffian system. Hence, it is reasonable to consider the complex 
$\Omega^\bullet(X)/\mJ^\bullet$, endowed with the differential 
induced by exterior differentiation on $\Omega^\bullet(X)$. 
This is well defined since $\mJ$ is stable by exterior differentiation. 
We define the \emph{characteristic cohomology} of $(X,W)$ to 
be the cohomology of this complex and we denote it by 
$H^\bullet_\mJ(X)$. More generally, we will attach the adjective 
\emph{characteristic} to the constructions related to the Pfaffian 
system. In \cite{BrGr1} and \cite{BrGr2}, this characteristic 
cohomology is intensively studied. Remark that if 
$f: Y \rightarrow X$ is a solution of $\mJ$, then one has a map 
in cohomology
$$
f^\ast : H_\mJ^\bullet(X) \rightarrow H^\bullet(Y).
$$

In \cite[\S III]{GrVHS}, the characteristic cohomology of period domains
is the object of an interesting conjecture. Period domains $D$ 
are homogeneous spaces $G/V$, where $V$ is 
a compact subgroup of a Lie group $G$, encountered 
in the study of variations of Hodge structures. 
To every variation of Hodge structures over a complex manifold $S$, 
one can construct a holomorphic map from $S$ to a quotient 
$\Gamma\backslash D$, where $\Gamma$ is a discrete subgroup 
of $G$, and the well-known \emph{Griffiths transversality condition} 
(a.k.a \emph{infinitesimal period relation}, see e.g. \cite{GrVHS2})
states that the differential of such a map has its values in 
some distribution $W$ of $\Gamma\backslash D$, coming from a
 $G$-invariant distribution of $D$. 
The conjecture can be stated as follows:

\begin{conj}\label{conjGr}
Let $\Gamma$ be a cocompact subgroup of $G$, acting freely on 
$D = G/V$. If the distribution $W$ is bracket-generating and if 
$m \leq m_0$, where $m_0$ is some integer determined by the 
Pfaffian system $\mJ$ associated to $W$, 
then $H_\mJ^m(\Gamma\backslash D)$ carries a natural 
pure real Hodge structure of weight $m$.
\end{conj}

A real pure Hodge structure of weight $m$ is a 
real vector space $E$ whose complexification $E \otimes_\R \C$ carries 
a decomposition
\begin{align*}
E \otimes_\R \C = \bigoplus_{p+q = m} E^{p,q},
\end{align*}
where $E^{p,q}$ are complex vector spaces satisfying 
$E^{p,q} = \overline{E^{q,p}}$. The prototype of real pure 
Hodge structure of weight $m$ is the real cohomology 
in degree $m$ of a compact K\"{a}hler manifold, as results from 
Hodge theory (especially the ellipticity of the Hodge Laplacian) 
and K\"{a}hler identities (see \cite{GriffithsHarris} or \cite{MM07}
for instance). One idea to study this conjecture is to develop 
an analogue of Hodge theory in this \emph{characteristic situation}. 
More precisely, following \cite{GrGrKe10}, one defines 
a Laplacian related to the Pfaffian system and one can try to prove that there 
is an isomorphism between its harmonic forms and the 
characteristic cohomology. In the complex setting, if the Laplacian 
respects the bigrading on forms, we get a Hodge structure on 
harmonic forms, hence on characteristic cohomology. However, 
we will see that the picture is not so bright and that 
Conjecture \ref{conjGr} certainly needs to be studied in another way.\\

In the first part, we construct this \emph{characteristic Laplacian}. 
In \cite[\S III. A]{GrGrKe10}, it is asserted that this characteristic 
Laplacian is hypoelliptic and that in this case we have 
an isomorphism of the characteristic cohomology with the space 
of harmonic forms. However, we explain why there seems to be no 
reason for the hypoellipticity of the characteristic Laplacian
in general and, in Example \ref{exataylorc}, we 
give an explicit counterexample to the hypoellipticity in a complex setting:

\begin{prop}
One can construct compact complex manifolds of dimension $3$, endowed with a
contact structure, for which the corresponding characteristic Laplacian is \emph{not}
hypoelliptic in degree $2$.
\end{prop}

In particular, it seems difficult to understand the characteristic cohomology \emph{via} harmonic forms.

In the second part, we nevertheless study the characteristic 
Laplacian in more details and answer the following question, 
asked in \cite[\S III. A]{GrGrKe10}:

\begin{qtn}\label{qtncond}
Let $(X,h)$ be a hermitian manifold, endowed with a holomorphic 
distribution $W$. Let $W_\R$ denote the underlying
real distribution in $TX$ and consider the Pfaffian system 
associated to $W_\R$. What are the necessary and sufficient 
conditions for the characteristic Laplacian to respect the 
bigrading on differential forms on $X$?
\end{qtn}
We give an unexpected answer to this question in Theorem \ref{thmprin}:

\begin{thm}
The characteristic Laplacian never respects the bigrading 
when the distribution $W$ is not involutive (in particular, when it is bracket-generating, as in Conjecture \ref{conjGr}).
\end{thm}

This is quite deceptive since it shows that there is nothing like a K\"{a}hler 
condition in the characteristic case. Indeed, in the classical case where the 
distribution $W$ is the whole space $TX$, we show in Theorem \ref{almt2.3}:

\begin{thm} A complex Hermitian manifold  is K\"{a}hler
if and only if the Hodge Laplacian preserves the bigrading 
on differential forms on $X$.
\end{thm}

It seems that the necessity part was not written yet in the literature.
 This result and its proof are independant of the rest of the article.

In this paper, if $E$ is a complex vector bundle on a manifold $X$, 
we will denote by $E_{\R}$ its underlying real vector bundle.
Moreover, $\wedge$ and $i$ are the exterior and interior products 
on $\Omega(X)$ and all distributions are assumed to be of 
constant rank.\\

\paragraph{\textbf{Acknowledgments}}
We thank Professors R. Bryant and P. Griffiths for useful 
discussions and for sending us M. Taylor's unpublished work. 
X. Ma thanks the Institut Universitaire de France for support.

\section{The characteristic Laplacian}\label{s1}

This section is organized as follows. In Subsection \ref{s1.1}, 
we define the characteristic Laplacian 
associated with a distribution for  Riemannian manifolds.
In Subsection \ref{s1.2},  we explain Taylor's counterexample 
for the hypoellipticity of the  characteristic Laplacian.
In Subsection \ref{princsymb},  we explain why 
the hypoellipticity does not seem to hold in general 
by computing its principal symbol.
In Subsection \ref{subscompl},  we give a counterexample
for the hypoellipticity of the  characteristic Laplacian
in the complex setting which is the context of
the original question.

\subsection{Definitions and notations}\label{s1.1}

Let $(X,g^{TX})$ be a smooth compact Riemannian manifold, 
endowed with a (constant-rank) distribution $W$. 
We denote by $F$ the annihilator of $W$ in $T^\ast X$; 
it is a vector subbundle of $T^\ast X$. We denote by 
$\Omega^\bullet(X)$ the graded algebra of differential forms on 
$X$ and we endow it with the natural metric $g^{\Omega(X)}$
induced by $g^{TX}$. We consider
\begin{itemize}
 \item $\mI$ the algebraic ideal generated by the smooth sections $I_X=\cC^{\infty}(X, F)$;
 \item $\mJ$ the differential ideal generated by the smooth 
 sections of $F$ on $X$, that is the minimal algebraic ideal 
 containing the smooth sections of $F$ and stable by exterior 
 differentiation.
\end{itemize}

Remark that $\mJ$ is the algebraic ideal generated by 
$I_X$ and $dI_X$. If $(\theta_{j})$ is a frame of $F$, 
the forms in $\mI$ can locally be written as
$$
\sum_j \theta_j \wedge \phi_j 
$$
where $\phi_j$ are arbitrary forms on $X$, and those in $\mJ$ 
are of the form
\begin{align*} 
\sum_j \theta_j \wedge \phi_j 
+ d\theta_j \wedge \psi_j 
\end{align*}
where $\phi_j, \psi_j$ are arbitrary forms on $X$.

Let $\mQ$ be the orthogonal complement of $\mJ$ in $\Omega(X)$. 
Remark that $\mQ$ is naturally graded. We define a differential 
operator for forms in $\mQ$ by
\begin{align} \label{cc1.3}
d_\mQ := \pi_\mQ \circ d\circ \pi_\mQ : \mQ\to \mQ,
\end{align}
where $\pi_\mQ$ is the orthogonal projection from $\Omega(X)$ 
onto $\mQ$.
Since $\mJ$ is stable by $d$, we have
\begin{align} \label{cc1.5}
\pi_\mQ \circ d\circ \pi_\mQ = \pi_\mQ \circ d: 
\Omega^{\bullet}(X)\to \Omega^{\bullet}(X),
\end{align}
By \eqref{cc1.5}, we know
\begin{align} \label{cc1.6}
d_\mQ^2= \pi_\mQ \circ d\circ \pi_\mQ\circ d\circ \pi_\mQ 
= \pi_\mQ \circ d^2\circ \pi_\mQ =0.
\end{align}
Since $\mQ$ is the orthogonal complement of $\mJ$ and $\mJ$ 
is stable by $d$, we know that $\mQ$ is stable by the adjoint 
$d^\ast$ of $d$ and the restriction 
of $d^\ast$ to $\mQ$ is $d_\mQ^*$, the adjoint of $d_\mQ$ 
for the natural $L^2$-structure on $\mQ$. Indeed, if $\alpha$ 
(resp. $\beta$) is in $\mQ$ (resp. $\mJ$), then $d\beta$ 
is in $\mJ$ and this implies 
$$
(d^\ast\alpha,\beta) = (\alpha, d\beta) = 0.
$$
Since this is true for any $\beta$ in $\mJ$, $d^\ast \alpha$ is in 
$\mQ$. Moreover, if $\alpha$ and $\beta$ are in $\mQ$, then 
\begin{equation} \label{equadj}
(d^\ast \alpha, \beta) = (\alpha, d\beta) = (\alpha, d_\mQ\beta)
\end{equation}
and this shows that $d_\mQ^\ast \alpha = d^\ast \alpha$.

\begin{defn} \label{sublaplace}
 The \emph{characteristic Laplacian} $\Delta_\mQ$ on $X$ 
 with respect to $W$ is the differential operator on $\mQ$
\begin{align} \label{cc1.10}
\Delta_\mQ =
d_\mQ d_\mQ^\ast + d_\mQ^\ast d_\mQ : \mQ \to \mQ.
\end{align}
\end{defn}

\begin{rem}
In sub-Riemmanian geometry (where the distribution $W$ is involutive),
 one defines a sub-Laplacian on functions (see \cite{Mont02}). 
 This sub-Laplacian is 
hypoelliptic and coincides with the characteristic Laplacian in degree $0$. 
In Example \ref{exataylor}, we will see that 
hypoellipticity can fail in positive degrees.
\end{rem}

Remember that we defined the characteristic cohomology of $X$ 
(associated to the distribution $W$) to be
\begin{align}
H^\bullet_\mJ(X) := H^\bullet(\Omega^\bullet(X)/\mJ^\bullet,d),
\end{align}
with the differential induced by exterior differentiation on 
$\Omega^\bullet(X)$.

By definition of $\mQ$ and $d_Q$, this characteristic cohomology 
is naturally isomorphic to the cohomology of the complex 
$(\mQ^\bullet, d_Q)$. An analogue of Hodge theory would 
be the following conjecture.

\begin{conj} \label{conj1} We denote by $\mH_\mQ^\bullet(X)$ 
 the \emph{characteristic harmonic forms}, that is the kernel of 
$\Delta_\mQ$. Then
\begin{itemize}
\item $\mH_\mQ^\bullet(X)$ is of finite dimension.
\item There is an orthogonal decomposition
\begin{align}
\mQ^\bullet(X) = \mH_\mQ^\bullet(X) \oplus d_\mQ(\mQ^{\bullet-1}(X)) 
\oplus d^\ast_\mQ(\mQ^{\bullet+1}(X)).
\end{align}
\item The natural application $\mH_\mQ^\bullet(X) 
\rightarrow H^\bullet_\mJ(X)$ is an isomorphism.
\end{itemize}
\end{conj}

In classical Hodge theory, one gets these results as consequences 
of the ellipticity of the Laplacian. In the next section, we show that 
the characteristic Laplacian in not even hypoelliptic in general. 
Hence we cannot really hope this conjecture to be true.

\subsection{The question of hypoellipticity}\label{s1.2}

We recall that, if $E$ and $F$ are vector bundles over $X$ and 
$P: E \rightarrow F$ is a differential operator, then $P$ is said 
to be \emph{hypoelliptic} if the following condition is satisfied: 
for every local distribution $u$ with values in $E$, if $Pu$ is 
smooth on an open set $U \subset X$, then the restriction 
of $u$ to $U$ is smooth. Elliptic operators, like the usual 
Hodge Laplacian, are hypoelliptic. It is a natural question to 
ask whether $\Delta_\mQ$ is hypoelliptic since this would 
be the first step in order to prove Conjecture \ref{conj1}.

The most known sufficient condition for a second order differential 
operator to be hypoelliptic is due to H\"{o}rmander (\cite{Hor}).

\begin{thm}[Sum of squares condition of hypoellipticity]\label{horm}
Let $P$ be a second order differential operator from a vector 
bundle $E$ to itself. Suppose that locally one can find smooth 
vector fields $X_0,\dots,X_k$ and a smooth function $c$ such that 
in a local frame of $E$,
\begin{align*}
Pu = (\sum_{i=1}^k X_i^2 + X_0 + c)u
\end{align*}
(in particular $P$ acts componentwise). 
Then $P$ is hypoelliptic if and only if $X_0,\dots,X_k$ generate 
$TX$ by brackets.
\end{thm}

In \cite{GrGrKe10}, it is suggested that this theorem implies that 
the characteristic Laplacian is hypoelliptic when the distribution 
$W$ is bracket-generating. We will first give a counterexample 
due to Michael Taylor \cite{Ta10} and then compute the principal 
symbol of $\Delta_\mQ$ in order to understand why the hypoellipticity
certainly fails in general.

\begin{exa}\label{exataylor}
A contact structure on a $3$-manifold $M$ is one of 
the simplest examples of Pfaffian systems. It is the datum of 
a $2$-rank distribution $W$ on $M$ which is bracket-generating. 
For instance, one can take for $W$ the kernel of the $1$-form 
$\theta = du - pdx$ in coordinates $(x,u,p)$ in $M = \R^3$.
A natural example in which $M$ is compact is constructed 
as follows: we consider $\mH^3$ the $3$-dimension Heisenberg 
group, that is $\R^3$ with coordinates $(p,q,t)$ and group 
structure $(p,q,t)\cdot(p',q',t') = (p+p',q+q',t+t'+\frac12(pq'-p'q))$. 
One checks that the $1$-form 
$\theta = dt - \frac12 qdp + \frac12 pdq$ is right-invariant and 
defines a contact structure on $\mH^3$. 
Taking a cocompact discrete subgroup $\Gamma$ of $G$, 
$M = G/\Gamma$ still carries the contact structure.

Locally, all contact structures on $3$-manifolds are the same. 
Let $M$ be a $3$-dimensional Riemannian manifold with a contact 
form $\theta$ and corresponding $2$-rank distribution $W$. 
Let $U$ be an open set in $M$. Then

\begin{itemize}
\item $\mJ^0(U) = 0$,
\item $\mJ^1(U) = \cC^\infty(U,\R\theta)$,
\item $\mJ^2(U) = \Omega^2(U)$, $\mJ^3(U) = \Omega^3(U)$,
\end{itemize}
and
\begin{itemize}
\item $\mQ^0(U) = \cC^\infty(U)$,
\item $\mQ^1(U) = \cC^\infty(U,E)$,
\item $\mQ^2(U) = 0$, $\mQ^3(U) = 0$,
\end{itemize}
where $E \rightarrow U$ is the real vector bundle of rank $2$, 
which is the orthogonal of $\R\theta$ in $T^\ast\R^3$. In degree $1$, 
the characteristic Laplacian is well-defined on forms with compact 
support. We have
$$
\Delta^1_{\mQ} = d_\mQ d^\ast_\mQ : \cC_c^\infty(U,E) 
\rightarrow \cC_c^\infty(U,E)
$$
and one can consider it as a second order differential operator 
on $E$. Let us denote by $(X_1,X_2)$ a (smooth) orthonormal 
frame of $W$ over $\ov{U}$ and by $(\alpha^1,\alpha^2)$
the dual frame of $W^\ast$. Using the metric, $W^\ast$ can 
be seen as a subbundle of $T^\ast M$ and $E$ can be
identified with $W^\ast$ (see also the following 
subsection \ref{princsymb}). 

A form $\mu \in L^2(U,E)$ can thus be written 
$\mu = \mu_1 \alpha^1 + \mu_2 \alpha^2$, with
$\mu_i \in L^2(U,\R)$, and 
$d^\ast_\mQ \mu$ is given by
\begin{align} \label{formreel}
d^\ast_\mQ \mu = Y_1 \mu_1 + Y_2 \mu_2,
\end{align}
where $Y_1,Y_2$ are first order scalar differential operators on $U$.

In particular
$$
d^\ast_\mQ (\mu_1 \alpha^1) = Y_1 \mu_1.
$$
We claim that $Y_1 \mu_1$ can be smooth (even zero) without 
$\mu_1$ being smooth. Indeed, in some local system of centered coordinates 
$(x_1,x_2,x_3)$, star-shaped in $0$, $Y_1$ has the form 
$$
Y_1 = \frac{\d}{\d x_1} + f(x_1,x_2,x_3),
$$
where $f$ is some smooth function. Consider the function 
$$v(x_1,x_2,x_3) = \exp\Big(\int_0^{x_1}f(t,x_2,x_3)dt\Big) 
\mu_1(x_1,x_2,x_3)$$
which is smooth if and only if $\mu_1$ is smooth. We compute that 
$$
\frac{\d}{\d x_1} v(x_1,x_2,x_3) = 
\exp\Big(\int_0^{x_1}f(t,x_2,x_3)dt\Big) Y_1 \mu_1(x_1,x_2,x_3).
$$
Choosing $v$ independant of $x_1$ but \emph{not} smooth, one has 
$Y_1 \mu_1 = 0$, proving the claim.

\end{exa}

\begin{rem}
In an analogous example occurring in the complex situation
(Example \ref{exataylorc}), we will need to be more precise.
In particular, formula (\ref{formreel}) can be made explicit:
\begin{align}
d^\ast_\mQ \mu = - \dive(\mu_1 X_1 + \mu_2 X_2),
\end{align}
where $\dive(X)$ is the divergence of the vector field $X$.
\end{rem}



In order to better understand why the hypoellipticity seems to fail, 
we will compute the principal symbol of $\Delta_\mQ$. 
It will not show that $\Delta_\mQ$
is not hypoelliptic but it will at least show that the sum of 
squares condition cannot be applied, at least if one only 
considers the second order terms.

\subsection{The principal symbol of the characteristic Laplacian}\label{princsymb}

Recall that we denote by $F$ the annihilator of $W$ in $T^\ast X$. 
Let $N$ be the orthogonal complement of $W$ in $(TX, g^{TX})$. 
Then as a smooth vector bundle, we have 
\begin{align} \label{cc1.12}
TX= W\oplus N, \quad \text{ and  } \quad 
T^\ast X = W^*\oplus N^{*}.
\end{align}
We can identify $N^*$ and $F$ as $\cC^\infty$ vector bundles, and 
\begin{align} \label{cc1.13}
\mI= \cC^\infty (X, N^* \widehat{\otimes} \Lambda (T^* X)) .
\end{align}

The quotient $\Omega(X)/\mI$  can be identified with the orthogonal 
complement of $\mI$, i.e.,
\begin{align} \label{cc1.14}
\Omega_W(X) := \cC^{\infty}(X, \Lambda W^\ast).
\end{align}
In what follows, we do these identifications  without further notice.

Since $\mI \subset \mJ$, $\mQ=\mJ^\bot$ can be viewed as 
a subspace of $ \Omega_W(X)$. The orthogonal complement of 
$\mQ$ in $\Omega_W(X)$ will be identified with $\mJ/\mI$. 
We thus have the following decompositions:
\begin{equation} \label{fst}
 \Omega(X)
 = \mI \oplus \Omega_W(X)
\end{equation}
and
\begin{equation}\label{fst1}
 \Omega_W(X)= \mJ/\mI \oplus \mQ.
\end{equation}

All of these spaces are naturally graded. We define a map 
$\varphi: F \to \Lambda^2(W^*)$ by: for 
$\theta\in \cC^\infty(X,F)$, $v,w\in \cC^\infty(X,W)$, set
\begin{align}\label{cc2.2}
    \varphi (\theta)(v,w) := (d\theta)(v,w) 
    =-  \theta([v,w]).
\end{align}
We check that for any $x\in X$, $\varphi (\theta)(v,w)_{x}$ 
depends only on $\theta_{x}$, $v_{x}$ and $w_{x}$.

The map $\varphi: F \to \Lambda^2 W^*$ induces a map 
$\varphi: F\widehat{\otimes}\Lambda^k W^* 
\to \Lambda^{k+2}W^*$ for any $k$. We will \textbf{assume} 
that the rank of these maps is constant on $X$, for any $k$. 
In particular, $\varphi(F\widehat{\otimes}\Lambda^k W^*)$ 
forms a vector subbundle of $\Lambda^{k+2} W^*$ on $X$ 
for any $k$. 
Set $F_{\varphi}:=\varphi(F\widehat{\otimes}\Lambda W^*)$ 
and let $F_{\varphi,\bot}$ be the orthogonal complement of 
$F_{\varphi}$ in $\Lambda W^*$ over $X$.

By construction, we thus have an orthogonal decomposition 
\begin{align}\label{cc2.3}
\Lambda W^\ast  = F_\varphi \oplus F_{\varphi,\bot}
\end{align}
and by (\ref{cc1.14}) and (\ref{cc2.3}), this decomposition
induces (\ref{fst1}), that is
\begin{equation}\label{cc2.40}
\begin{aligned}
 \mJ/\mI &= \cC^\infty(X,F_\varphi),\\
 \mQ &= \cC^\infty(X,F_{\varphi,\bot}).
\end{aligned}
\end{equation}
We denote by $\pi_{F_\varphi}$ and $\pi_{F_{\varphi,\bot}}$ 
the orthogonal projections from $\Lambda W^\ast$ onto 
$F_{\varphi}$ and $F_{\varphi,\bot}$. In order to make 
computations with the operators on $\mQ$, 
we construct intermediate operators on $\Lambda W^\ast$. 
First we define 
\begin{align} \label{cc1.15}
d_W := \pi_W \circ d \circ \pi_W : \Lambda W^\ast  
\rightarrow \Lambda W^\ast ,
\end{align}
where $\pi_W$ is the projection from $\Lambda(T^\ast X)$ 
on $\Lambda W^\ast$ in the decomposition (\ref{fst}). 
Beware that there is no reason for $d_W^2$ to be $0$. 
The adjoint $d_W^\ast$ of $d_W$ satisfies
\begin{align} \label{cc1.16}
d_W^\ast = \pi_W \circ d^\ast \circ \pi_W.
\end{align}

From (\ref{cc1.3}), (\ref{cc1.15}) and (\ref{cc1.16}), we have
\begin{equation}\label{cc1.18}
\begin{aligned} 
d_\mQ &= \pi_\mQ \circ d_W \circ \pi_\mQ,\\
d_\mQ^\ast &= \pi_\mQ \circ d_W^\ast \circ \pi_\mQ.
\end{aligned}
\end{equation}

By (\ref{cc1.15}) and (\ref{cc1.16}), $d_W$ and $d_W^\ast$ 
are first order differential operators on $\Lambda W^\ast$ 
and their principal symbols are, for $\xi \in T^\ast X$

\begin{equation}
\begin{aligned}
\sigma_1 (d_W, \xi) &= \sqrt{-1}\, \xi_W \wedge,\\
\sigma_1 (d_W^\ast, \xi) &= -\sqrt{-1}\, i_{\xi_W^\ast},
\end{aligned}
\end{equation}
where $\xi_W$ is the orthogonal projection of $\xi$ on $W^\ast$ 
and $\xi_W^\ast \in W$ is the metric dual of $\xi_W$.

By (\ref{cc2.40}) and (\ref{cc1.18}), $d_\mQ$ and $d_\mQ^\ast$ 
are first order differential operators on $F_{\varphi,\bot}$ 
and their principal symbols are
\begin{equation} \label{symbd}
\begin{aligned}
\sigma_1(d_\mQ,\xi) &= \sqrt{-1}\,\pi_{F_{\varphi,\bot}} 
\xi_W \wedge \pi_{F_{\varphi,\bot}},\\
\sigma_1(d_\mQ^\ast,\xi) &= -\sqrt{-1}\,\pi_{F_{\varphi,\bot}} 
i_{\xi_W^\ast} \pi_{F_{\varphi,\bot}}.
\end{aligned}
\end{equation}

One also gets the adjoint formula of (\ref{cc1.5})
\begin{align} \label{cc1.5b}
\pi_\mQ \circ d^\ast \circ \pi_\mQ = d^\ast \circ \pi_\mQ.
\end{align}

Taking the principal symbols of (\ref{cc1.5}) and (\ref{cc1.5b}), 
we have
\begin{equation}  \label{cc2.10}
\begin{aligned}
\pi_{F_{\varphi,\bot}}   \xi_W \wedge\pi_{F_{\varphi,\bot}} 
= \pi_{F_{\varphi,\bot}} \xi_W \wedge,\\
\pi_{F_{\varphi,\bot}}    i_{  \xi_W ^*}  \pi_{F_{\varphi,\bot}} 
= i_{  \xi_W ^*}  \pi_{F_{\varphi,\bot}}.
\end{aligned}
\end{equation}

\begin{prop} \label{propprinc}
The characteristic Laplacian is a second order differential operator 
on $F_{\varphi,\bot}$ and its principal symbol is
\begin{align}\label{symblap}
\sigma_2(\Delta_\mQ,\xi) = \pi_{F_{\varphi,\bot}}
(|\xi_W|^2 - i_{\xi_W^\ast} \pi_{F_\varphi} \xi_W \wedge) 
\pi_{F_{\varphi,\bot}}.
\end{align}
\end{prop}

\begin{rem}
The first term in (\ref{symblap}) is the term suggested in 
\cite{GrGrKe10} but the second term was forgotten. 
Because of this second term, one cannot apply H\"{o}rmander's 
condition of hypoellipticity; see Example \ref{exataylorb}.
\end{rem}

\begin{proof}[Proof of Proposition \ref{propprinc}]
By (\ref{cc1.10}), (\ref{symbd}) and (\ref{cc2.10}), 
$\Delta_\mQ$ is a second order differential operator on 
$F_{\varphi,\bot}$ and its principal symbol is
\begin{align*}
\sigma_2(\Delta_\mQ,\xi) &= \pi_{F_{\varphi,\bot}} \xi_W 
\wedge \pi_{F_{\varphi,\bot}} i_{\xi_W^\ast} 
\pi_{F_{\varphi,\bot}} +
 \pi_{F_{\varphi,\bot}} i_{\xi_W^\ast} \pi_{F_{\varphi,\bot}}
 \xi_W \wedge \pi_{F_{\varphi,\bot}}\\
&= \pi_{F_{\varphi,\bot}} \xi_W \wedge i_{\xi_W^\ast} 
\pi_{F_{\varphi,\bot}} + 
\pi_{F_{\varphi,\bot}} i_{\xi_W^\ast} \xi_W \wedge 
\pi_{F_{\varphi,\bot}} - \pi_{F_{\varphi,\bot}} i_{\xi_W^\ast} 
\pi_{F_{\varphi}} \xi_W \wedge \pi_{F_{\varphi,\bot}}\\
&= \pi_{F_{\varphi,\bot}}(|\xi_W|^2 - i_{\xi_W^\ast} 
\pi_{F_\varphi} \xi_W \wedge) \pi_{F_{\varphi,\bot}}.
\end{align*}

In the second equality, we use (\ref{cc2.10}) and in the third, 
we use the identity 
$$i_{\xi_W^\ast} \xi_W \wedge + \xi_W \wedge i_{\xi_W^\ast} 
= |\xi_W|^2.$$

The proof of Proposition \ref{propprinc} is completed.
\end{proof}

\begin{exa}\label{exataylorb}
We review Example \ref{exataylor} and compute the principal 
symbol of the characteristic Laplacian in degree one. 
With the identifications at the beginning of the paragraph, 
$\mQ^1$ is the space of sections of $W^\ast$ and $\mQ^2$
is zero. Let $\eta$ be in $W^\ast$. Since in (\ref{symblap}), 
only the projection $\xi_W$ of $\xi$ is involved, we can 
restrict the symbol to the $\xi$ belonging to $W^\ast$. We have
\begin{align*}
\sigma_2(\Delta_\mQ, \xi) \eta &= |\xi|^2 \eta - 
\pi_{F_{\varphi,\bot}} i_{\xi^\ast} 
\pi_{F_\varphi} (\xi \wedge\eta)\\
&= |\xi|^2 \eta - i_{\xi^\ast}(\xi \wedge \eta)\\
&= |\xi|^2 \eta - (|\xi|^2 \eta - \eta(\xi^\ast)\xi)\\
&= \eta(\xi^\ast)\xi.
\end{align*}
In the second equality, we use that $F_{\varphi,\bot}$ 
(resp. $F_\varphi)$ equals $\Lambda W^\ast$ in 
degree $1$ (resp. $2$).

If $\xi = (\xi_1, \xi_2)^t$ and $\eta = (\eta_1, \eta_2)^t$
in an orthonormal frame for $W^\ast$, we get
\begin{align*}
\sigma_2(\Delta_\mQ, \xi) \binom{\eta_1}{\eta_2} 
= (\xi_1\eta_1 + \xi_2\eta_2)\binom{\xi_1}{\xi_2} 
= \binom{\xi_1^2 \eta_1 
+ \xi_1\xi_2 \eta_2}{\xi_1\xi_2 \eta_1 + \xi_2^2 \eta_2}.
\end{align*}

Otherwise said, we have the equality
\begin{align*}
-\Delta_\mQ \binom{\eta_1}{\eta_2} 
= \binom{X_1^2 \eta_1 + X_1X_2\eta_2}{X_1X_2 \eta_1 
+ X_2^2 \eta_2} + X_0 \binom{\eta_1}{\eta_2},
\end{align*}
where $(X_1,X_2)$ is a local frame of $W$ 
and $X_0$ is a first order differential 
operator, which does not necessarily act componentwise. 
Since the second order does not act componentwise, 
one cannot apply H\"{o}rmander's condition of Theorem \ref{horm}.
\end{exa}
\subsection{The complex situation}\label{subscompl}

In the remainder of the article, we will be interested 
in Pfaffian systems over a complex manifold. More precisely, 
let $(X,J)$ be a compact complex manifold; $J$ induces a splitting 
$TX\otimes_\R \C=T^{(1,0)}X\oplus T^{(0,1)}X$, 
where $T^{(1,0)}X$ and $T^{(0,1)}X$
\index{$T^{(1,0)}X, T^{(0,1)}X$} are the eigenbundles 
of $J$ corresponding to the eigenvalues $\sqrt{-1}$ and 
$-\sqrt{-1}$, respectively. Let $T^{*(1,0)}X$  and $T^{*(0,1)}X$ 
\index{$T^{*(1,0)}X, T^{*(0,1)}X$}
be the corresponding dual bundles.

We still denote by $\Omega^k(X)$ the space of smooth $k$-forms 
on $X$ with values in $\C$. Let 
\begin{align} \label{lm2.11b}
\Lambda^{p,q}(T^*X)=
\Lambda^p(T^{*(1,0)}X)\otimes \Lambda^q(T^{*(0,1)}X),\quad
\Omega^{p,q}(X):
=\cC^\infty(X,\Lambda^{p,q}(T^*X)).
\end{align}
Then $\Omega^{p,q}(X)$ is the space of smooth  
$(p,q)$-forms on $X$, and  
$\Omega^k(X)=\oplus_{p+q=k}\, \Omega^{p,q}(X)$.

Let $\Theta$ be a real $(1,1)$-form such that
 \begin{align} \label{lm02.12}
g^{TX}(\cdot,\cdot ) = \Theta(\cdot,J\cdot)
\end{align}
defines a Riemannian metric on $TX$. The triple $(X,J, \Theta)$
is called a \emph{complex {H}ermitian manifold}. If $\Theta$ is 
a closed form, then the form $\Theta$ is called a K\"{a}hler form 
on $X$.

We denote the holomorphic tangent bundle by $T_h X$. 
Let $W \subset T_h X$ be a (constant-rank) holomorphic distribution. 
We consider the Pfaffian system associated 
to the distribution $W_\R$. Otherwise said, if we denote by 
$F \subset T_h^\ast X$ the holomorphic annihilator of $W$, 
then the exterior differential system $\mJ$ we consider 
is generated by $F_\R \subset T^\ast X$. 
Beware of the notations that differ from the real case. 
Remark that $\mJ$ is not only $d$-stable: it is also $\d$ 
and $\bar{\d}$-stable. Indeed, if $I_X=\cC^{\infty}(X, F)$
is the space of smooth sections of $F$ on $X$ then 
$\ov{I}_X=\cC^{\infty}(X,  \ov{F})$ and $d$ acts on $I_X$ 
(resp. $\ov{I}_X$) as $\d$ modulo 
$\ov{I}_X \cdot \Omega^{1}(X)$ (resp. $\bar{\d}$
modulo $I_X  \cdot \Omega^{1}(X)$).

Locally, if $(\theta_{j})$ is a holomorphic frame of $F$, 
the forms in $\mI$ can be written as
$$
\sum_j \theta_j \wedge \psi_j 
+ \overline{\theta}_j \wedge \phi_j,
$$
where $\psi_j, \phi_j$ are arbitrary forms, and those in $\mJ$ 
are of the form
\begin{align*}
\sum_j \theta_j \wedge \psi_j 
+ \overline{\theta}_j \wedge \phi_j
+ d\theta_j \wedge \omega_j 
+ d\overline{\theta}_j \wedge \chi_j,
\end{align*}
where $\psi_j, \phi_j, \omega_j, \chi_j$ are arbitrary forms.

We still denote by $\mQ$ the orthogonal of $\mJ$. Besides 
the operator $d_\mQ$, we also define $\d_\mQ$ and 
$\bar{\d}_\mQ$ by
\begin{align} \label{cc1.4}
\d_\mQ := \pi_\mQ \circ \d\circ \pi_\mQ, \quad
\bar{\d}_\mQ := \pi_\mQ \circ \bar{\d}\circ \pi_\mQ.
\end{align}

Moreover, as in (\ref{equadj}), the adjoints of $\d_\mQ$ and $\bar{\d}_\mQ$ 
(for the natural $L^2$-structure on $\mQ$) are the restrictions 
to $\mQ$ of $\d^\ast, \bar{\d}^\ast$ the adjoints of $\d$ and $\bar{\d}$.

If we denote by $N$ the orthogonal complement of $W$
in $(TX,g^{TX})$, one obtains analogue decompositions 
as the ones in the previous subsection. In particular, 
denoting by $\mI$ the algebraic ideal generated by $F_\R$, 
one has the analogues of (\ref{cc1.13}), (\ref{cc1.14}), (\ref{fst}) 
and (\ref{fst1}):
\begin{equation} \label{cc1.16a}
\begin{aligned}
&\mI = \cC^\infty (X, N_\R^* \widehat{\otimes} \Lambda (T^* X)),\\
&\Omega_W(X) := \cC^{\infty}(X, \Lambda W_\R^\ast),\\
 &\Omega(X) = \mI \oplus \Omega_W(X),\\
 &\Omega_W(X) = \mJ/\mI \oplus \mQ.
\end{aligned}
\end{equation}

All of these spaces carry a natural bigrading. Denoting by 
$\pi_W$ the orthogonal projection from $\Lambda(T^\ast X)$ 
onto $\Lambda W_\R^\ast$, we define
\begin{align} \label{cc1.17}
\begin{split}
\d_W :=    \pi_W \circ \d \circ \pi_W,  \quad\bar{\d}_W:= 
\pi_W \circ  \bar{\d} \circ \pi_W,\\
\d_W^\ast =    \pi_W \circ \d^\ast \circ \pi_W,  \quad
\bar{\d}_W^\ast=  \pi_W \circ  \bar{\d}^\ast \circ \pi_W.
\end{split}
\end{align}
Then, we have, besides equations (\ref{cc1.15}),
(\ref{cc1.16}) and (\ref{cc1.18}),
\begin{align} \label{cc1.18a}
\begin{split}
 \d_\mQ =    \pi_\mQ \circ \d_W \circ \pi_\mQ,  \quad
\bar{\d}_\mQ=  \pi_\mQ \circ  \bar{\d}_W \circ \pi_\mQ,\\
\d_\mQ^\ast =    \pi_\mQ \circ \d^\ast_W \circ \pi_\mQ,  \quad
\bar{\d}_\mQ^\ast=  \pi_\mQ \circ  \bar{\d}^\ast_W\circ \pi_\mQ.
\end{split}
\end{align}

We still define a map $\varphi : F \rightarrow \Lambda^2 W^\ast$ 
as in (\ref{cc2.2}). Since in the definition (\ref{cc2.2}), 
we can take $\theta, v, w$ to be holomorphic sections, 
it proves that $\varphi$ is a holomorphic map. 
We still \textbf{assume} that the induced map 
$\varphi: F\widehat{\otimes}\Lambda W_\R^\ast 
\rightarrow \Lambda W_\R^\ast$ has constant rank. 
Set $F_\varphi :=\varphi(F\widehat{\otimes}\Lambda W^*)$ 
and $F_{\varphi,\bot}$ the orthogonal complement of 
$F_{\varphi}$ in $\Lambda W^*$ and let $\pi_{F_{\varphi}}$, 
$\pi_{F_{\varphi,\bot}}$ be the orthogonal projections from 
$\Lambda W^*$ onto $F_{\varphi}$, $F_{\varphi,\bot}$. 

Then $F_{\varphi}\widehat{\otimes}\Lambda \ov{W}^*
+\Lambda W^*\widehat{\otimes} \ov{F}_{\varphi}$ and 
$F_{\varphi,\bot}\widehat{\otimes} \ov{F}_{\varphi,\bot}$ 
are vector subbundles of $\Lambda W^*_{\R} \otimes_\R \C$ over  $X$, and  
\begin{align}\label{cc2.3b}\begin{split}
\Lambda W^*_{\R} \otimes_\R \C& =   
\Big(F_{\varphi}\widehat{\otimes}\Lambda \ov{W}^*
+\Lambda W^*\widehat{\otimes} \ov{F}_{\varphi}\Big)
\oplus F_{\varphi,\bot}\widehat{\otimes} \ov{F}_{\varphi,\bot}.
    \end{split}\end{align}
		
In this decomposition, we have, as in (\ref{cc2.40}),
\begin{align}\label{cc2.4}\begin{split}
\mJ/\mI &= \cC^\infty (X, F_{\varphi}\widehat{\otimes}\Lambda \ov{W}^*
+\Lambda W^*\widehat{\otimes} \ov{F}_{\varphi}). \\
\mQ &= \cC^\infty (X, F_{\varphi,\bot}\widehat{\otimes} 
\ov{F}_{\varphi,\bot}).
    \end{split}\end{align}
		
\begin{exa}\label{exataylorc}

Example \ref{exataylor} can also be seen in the complex 
situation but an interesting phenomenon appears in degree $1$. 
Consider the complex manifold $ M= \C^3$ with complex 
coordinates $(x,u,p)$ and the holomorphic $1$-form 
$\theta = du - pdx$. We denote by $(\cdot,\cdot)$ 
a Hermitian metric on $M$ and by $|\cdot|$ the corresponding 
norm. Using the notations of this subsection, 
$W$ is a holomorphic vector subbundle of $T_h X$ of rank $2$.
Hence $\Lambda^2 W^\ast $ is of rank $1$ and we have 
$F_\varphi = \Lambda^2 W^\ast $. We thus obtain the 
following bidegree decomposition of $\mQ$ over an open set $U$:

\begin{itemize}
\item $\mQ^{0,0}(U) = \cC^\infty(U)$,
\item $\mQ^{1,0}(U) = \cC^\infty(U,W^\ast)$,
\item $\mQ^{0,1}(U) = \cC^\infty(U, \overline{W}^\ast)$,
\item $\mQ^{1,1}(U) = \cC^\infty(U, W^\ast \widehat{\otimes}
\overline{W}^\ast)$.
\end{itemize}

Take $(X_1,X_2)$ a holomorphic frame of $W$ on $\overline{U}$ and 
$(\alpha^1,\alpha^2)$ its dual frame. We study the 
smoothness of harmonic forms for $\Delta_\mQ$ in degrees $0$, $1$ and $2$.

For $f \in L^2(U,\C)$, one has $\Delta_\mQ f = 0$ if and only if $d_\mQ f = 0$. 
Using the frames, one obtains
$$ d_\mQ f = X_1(f) \alpha^1 + X_2(f) \alpha^2 
+ \overline{X}_1(f) \overline{\alpha}^1 + \overline{X}_2(f)
\overline{\alpha}^2.$$
Since $X_1, X_2, \overline{X}_1$ and $\overline{X}_2$ 
generate by brackets the tangent bundle of $U$, 
a function $f$ such that $d_\mQ f = 0$ is in fact locally constant. 
Remark that the same argument works for any distribution $W$ 
which is bracket-generating. This proves that the harmonic
functions are smooth.

In degree $1$, let $\mu = \mu_1 \alpha^1 + \mu_2 \alpha^2$ 
be in $L^2(U,W^\ast)$ (here $\mu_1, \mu_2 \in L^2(U,\C)$). 
Such a form is cancelled by $\Delta^1_\mQ$ 
if and only if it is cancelled by both $d_\mQ$ and $d^\ast_\mQ$.

We denote by $\dive(Y)$ the divergence of a vector field $Y$ and
by $dv_X$ the Riemannian volume form on $(X,g^{TX})$.
\begin{lemma} \label{lemdeg1}
The following identities hold:
\begin{eqnarray*}
d_\mQ \mu &= \sum_{i,j = 1}^2 \overline{X}_j(\mu_i) 
\overline{\alpha}^j \wedge \alpha^i, \\
d^\ast_\mQ \mu &= - \dive(\sum_{i=1}^2 (\mu,\alpha_i) 
\overline{X_i}).
\end{eqnarray*}
\end{lemma}
\begin{proof}
Since $\mQ^{2,0}(U) = 0$, we know $d_\mQ \mu = \bar{\d}_\mQ \mu$. 
The $2$-form $d\alpha^k$ satisfies
$$
d\alpha^k(Y_1,Y_2) = Y_1(\alpha^k(Y_2))-Y_2(\alpha^k(Y_1)) 
- \alpha^k([Y_1,Y_2]).
$$
Since $(\alpha^1,\alpha^2)$ is the dual basis of $(X_1,X_2)$, 
this simplifies to $d\alpha^k(Y_1,Y_2) = -\alpha^k([Y_1,Y_2])$ 
when $Y_l$ is either a $X_i$ or a $\overline{X}_j$. Since
the $X_i$ are holomorphic, the brackets $[X_i,\overline{X}_j]$ 
vanish and $d\alpha^k(X_i,\overline{X}_j) = 0$. This implies that
\begin{eqnarray*}
d_\mQ \mu &= \sum_{i=1}^2 \bar{\d}_\mQ \mu_i \alpha^i
&= \sum_{i,j=1}^2 \overline{X}_j(\mu_i) \overline{\alpha}^j 
\wedge \alpha^i.
\end{eqnarray*}

For the second equality, one has for every smooth function $f$,
\begin{align*}
(d^\ast_\mQ \mu, f) &= (\d^\ast_\mQ \mu, f) \\
&= (\mu, \d_\mQ f) \\
&= - \sum_{i=1}^2 \int_M \overline{X_i(f)} (\mu,\alpha^i) dv_X \\
&= - \sum_{i=1}^2 \int_M \bar{f}(\overline{X}_i (\mu,\alpha^i) 
+ \dive \overline{X}_i) dv_X.
\end{align*}

Hence $d^\ast_\mQ \mu = -\sum_{i=1}^2 
(\overline{X}_i (\mu,\alpha^i) + (\mu,\alpha^i)\dive \overline{X}_i) 
= - \dive(\sum_{i=1}^2 (\mu,\alpha_i) \overline{X_i})$.
\end{proof}

Using Lemma \ref{lemdeg1}, if $\Delta^1_\mQ \mu= 0$ then 
$\overline{X}_j(\mu_i) = 0$ for all $i,j$ from $1$ to $2$. 
Since the $\overline{X}_j$ generate by brackets 
the anti-holomorphic tangent bundle, this is equivalent to 
$\mu_i$ being holomorphic in the weak sense. Since $\bar{\d}$ 
is an elliptic operator on functions, the $\mu_i$ are holomorphic.
This shows that the harmonic forms of bidegree $(1,0)$ 
(resp. $(0,1)$) are holomorphic (resp. anti-holomorphic) 
sections of $W^\ast$ (resp. $\overline{W}^\ast$). In particular, 
they are smooth, contrary to what happened in the real setting.

In degree $2$, consider a form $\nu = \sum_{i,j = 1}^2 
\nu_{ij} \overline{\alpha}^j \wedge \alpha^i$ in $L^2(U,W^\ast \widehat{\otimes}
\overline{W}^\ast$). It is cancelled by 
$\Delta^2_\mQ$ if and only if $d^\ast_{\mQ} \nu = 0$.

\begin{lemma} \label{lemdeg2}
The $2$-form $\nu$ is harmonic if and only if for $k=1,2$,
\begin{eqnarray*}
\dive \left(\sum_{l=1}^2 (\nu, \overline{\alpha}^l \wedge \alpha^k) X_l \right) 
&= 0,\\
\dive \left(\sum_{l=1}^2 (\nu, \overline{\alpha}^k \wedge \alpha^l) 
\overline{X}_l\right) &= 0.
\end{eqnarray*}
\end{lemma}

\begin{proof}
We just show that the first equation is equivalent to 
the vanishing of $\bar{\d}^\ast_\mQ \nu$. 
For every $(1,0)$-form $\mu = \mu_1 \alpha^1 + \mu_2 \alpha^2$,
\begin{align*}
(\bar{\d}^\ast_\mQ \nu, \mu) &= (\nu, \bar{\d}_\mQ \mu) \\
&= (\nu, \sum_{k,l = 1}^2 \overline{X}_l(\mu_k) 
\overline{\alpha}^l \wedge \alpha^k) 
\text{ by Lemma \ref{lemdeg1}} \\
&= \sum_{k,l=1}^2 \int_M X_l(\overline{\mu}_k)
(\nu,\overline{\alpha}^l \wedge \alpha^k) dv_X \\
&= -\sum_{k,l=1}^2 \int_M \overline{\mu}_k 
\dive((\nu,\overline{\alpha}^l \wedge \alpha^k)X_l) dv_X.
\end{align*}
This is zero for all $\mu$ if and only if $\dive(\sum_{l=1}^2 
(\nu, \overline{\alpha}^l \wedge \alpha^k) X_l)$ is $0$ 
for $k=1,2$.
\end{proof}

It seems difficult to unravel these equations in general. 
We will only consider two different choices of the metric.

\paragraph{\textbf{Standard metric}} In the particular case 
where the metric on $\C^3$ is the standard one, one can 
choose for $X_1,X_2$ orthogonal holomorphic vectors with 
zero divergence (take $X_1 = \frac{\d}{\d p}$ and 
$X_2 = p\frac{\d}{\d u} + \frac{\d}{\d x}$. Moreover, 
$X_1$ has norm $1$. Then the equations of Lemma \ref{lemdeg2} 
become
\begin{eqnarray*}
\sum_l X_l (\nu_{lk} |\overline{\alpha}^l \wedge \alpha^k|^2) = 0,\\
\sum_l \overline{X}_l( \nu_{kl} |\overline{\alpha}^k 
\wedge \alpha^l|^2) = 0,
\end{eqnarray*}
where $\nu = \nu_{lk} \overline{\alpha}^l \wedge \alpha^k$. 
One can take $\nu_{12} = \nu_{21} = \nu_{22} = 0$. 
Then the equations are simply 
$X_1 \nu_{11} = \overline{X}_1 \nu_{11} = 0$. 
Since $X_1$ is holomorphic, $[X_1,\overline{X}_1] = 0$ and, 
by Frobenius theorem, one can choose $\nu_{11}$ constant
in the directions of $X_1$ and $\overline{X}_1$ but $\nu_{11}$ 
not smooth. This shows the existence of non-smooth harmonic 
$2$-forms. \\

\paragraph{\textbf{Heisenberg metric}} Consider the case where
we see $\C^3$ as the complex Heisenberg group 
(see Example \ref{exataylor}), endowed with a right-invariant
Hermitian metric and with a right-invariant contact form. 
Choose a basis $(X'_1,X'_2)$ of $W$ 
at the identity $e$ of $\mH^3$. Consider the corresponding
right-invariant vector fields on $\mH^3$, denoted by 
$\widetilde{X'_1}$ and $\widetilde{X'_2}$. Since these vector fields 
and the volume form are right-invariant, the divergences of 
$\widetilde{X'_1}$ and $\widetilde{X'_2}$ are constant. 
Hence, to certain linear combination $X_1$ of $X'_1$ 
and $X'_2$ corresponds a right-invariant vector field 
with zero divergence. We can moreover assume that $X_1$ 
is a unit vector and complete it to an orthonormal basis 
$(X_1,X_2)$ of $W_e$. Thus we get a holomorphic 
orthonormal frame ($\widetilde{X_1},\widetilde{X_2})$ of $W$
and $\widetilde{X_1}$ has zero divergence. Then, 
the same argument as above shows the existence 
of non-smooth harmonic $2$-forms.
\end{exa}

\begin{rem}
This last example with the Heisenberg metric is in fact the prototype 
of the situation described in Conjecture \ref{conjGr}. 
Indeed, take $\Gamma$ a cocompact subgroup of the complex 
Heisenberg group. Then, by invariance of the metric and of the 
distribution,  we obtain a counterexample to the hypoellipticity 
of the characteristic Laplacian in a compact complex contact $3$-manifold.
 Those manifolds are the simplest examples of period domains, 
 with non-trivial Griffith's transversality condition (that is, 
 with a distribution not equal to the whole tangent space).
\end{rem}

\section{Answer to Question \ref{qtncond}}\label{s2}

The aim of this section is to prove the following theorem, 
which is an answer to question \ref{qtncond}:

\begin{thm}\label{thmprin}
In the notations of Question \ref{qtncond} and Section \ref{s1}, the characteristic 
Laplacian never respects the bigrading on $\mQ^{\bullet}$ 
when the distribution $W$ is not involutive.
\end{thm}

This section is organized as follows. In Subsection \ref{s2.1}, 
we show that a complex Hermitian manifold  is K\"{a}hler
if and only if the Hodge Laplacian preserves the bigrading 
on $\Omega(X)$. In Subsection \ref{s2.2}, 
we establish a generalized sub-K\"{a}hler identity.
In Subsection  \ref{s2.3},  we establish Theorem \ref{thmprin}.

\subsection{The classical case}\label{s2.1}

First we study the case where the distribution $W$ is the whole 
tangent space $TX$, which is interesting for itself. We thus have 
$\mQ = \Omega(X)$ and the characteristic Laplacian is the usual 
Hodge Laplacian, which we simply denote by $\Delta$. 
Remark that Theorem \ref{thmprin} says nothing in this case.

It is well known that for a K\"{a}hler manifold, its Hodge Laplacian 
preserves the bigrading of the differential forms 
(cf. \cite[\S0.7]{GriffithsHarris}, \cite[Corollary 1.4.13]{MM07}). 
This implies the decomposition of the complex valued 
de-Rham cohomology in bidegree type for 
a compact K\"{a}hler manifold; this was in fact the initial interest of 
the authors for the general question \ref{qtncond}. 
In \cite[\S III. A]{GrGrKe10}, 
Green, Griffiths and Kerr claimed that Chern \cite{Chern57} 
proved that for Hermitian manifolds,  if its Hodge Laplacian 
preserves the bigrading of the differential forms, 
then the Hermitian metric is K\"{a}hler. 
After communications with Professors Bryant and Griffiths, 
we realized that Chern did not claim this result in his paper 
\cite{Chern57}, and it seems that one could not find a proof 
in the literature. 

\begin{thm}\label{almt2.3} 
The complex Hermitian manifold $(X,J, \Theta)$ is K\"{a}hler
if and only if $\Delta$ preserves the bigrading on $\Omega(X)$,
i.e., $\Delta$ sends $(p,q)$-forms to $(p,q)$-forms.
\end{thm}

We first introduce some notations from \cite{MM07}.

For any $\Z_2$-graded vector space $V= V^+\oplus V^-$, 
the natural $\Z_2$-grading on $\End(V)$ is defined by
$$\End(V)^+= \End(V^+)\oplus \End(V^-),\quad
\End(V)^-= \Hom(V^+,V^-)\oplus \Hom(V^-,V^+),
$$
and we define $\deg B=0$ for $B\in\End(V)^+$, 
and $\deg B=1$ for $B\in\End(V)^-$. For $B,C\in \End(V)$, 
we define their supercommutator (or graded Lie bracket) by
\begin{equation}\label{alm1.35}
[B,C]=BC-(-1)^{\deg B\cdot \deg C}CB.
\end{equation}
Then for $B, B^\prime, C\in \End(V)$, the Jacobi identity holds:
\begin{multline} \label{jacobi}
(-1)^{\deg C\cdot \deg B^\prime}\big[B^\prime,[B,C]\big]
+(-1)^{\deg B^\prime \cdot\deg B}\big[B,[C,B^\prime]\big]\\
+(-1)^{\deg B\cdot\deg C}\big[C,[B^\prime,B]\big]=0.
\end{multline}
We will apply the above notation for $\Omega^{\bullet}(X)$ 
with natural $\Z_2$-grading induced by the parity of the degree,
(cf. \cite[(1.3.31)]{MM07}).

We define the Lefschetz operator $L=\Theta\,\wedge$ on 
$\Lambda^{\bullet,\bullet}(T^*X)$ and its adjoint $\Lambda=i(\Theta)$ with 
respect to the Hermitian product 
$\langle \cdot,\cdot\rangle_{\Lambda^{\bullet,\bullet}}$
induced by $g^{TX}$. For $\{w_j\}_{j=1}^{m}$ a local orthonormal 
frame of $T^{(1,0)}X$, we have
\begin{equation}\label{herm15,1}
L=\sqrt{-1}  \sum_{j=1}^{m}w^j\wedge \ov{w}^j\wedge\,,\quad 
\Lambda=-\sqrt{-1} \sum_{j=1}^{m} i_{\ov{w}_j}i_{w_j}\,,
\end{equation}
where $\wedge$ and $i$ denote the exterior and interior product, 
respectively. The  Hermitian torsion operator is defined by
\begin{equation}\label{herm18}
\mathcal{T}:=[\Lambda,\partial\Theta]
=[i(\Theta),\partial\Theta]\,.\index{$\mathcal{T}$}
\end{equation}

\begin{proof}[Proof of Theorem \ref{almt2.3}]
If  $(X,J, \Theta)$ is K\"{a}hler, then it is a classical result that 
$\Delta$ preserves the bigrading on $\Omega^\bullet(X)$, 
cf. for example \cite[Corollary 1.4.13]{MM07} for a proof.

We assume now that $\Delta$ preserves the bigrading on 
$\Omega^\bullet(X)$.

Let $\ov{\square}:=\partial\partial^*+\partial^*\partial$;
$\square:=\ov{\partial}\,\ov{\partial}^*+\ov{\partial}^*\ov{\partial}$ 
be the usual $\partial$-Laplacian and $\ov{\partial}$-Laplacian. 
Then as $d=\partial+ \db$ and $d^2=0$, 
 we have (cf. \cite[1.4.50)]{MM07}) 
\begin{equation}\label{lm2.65}
\Delta=[d,d^*]=[\partial+\db,\partial^*+\db^*]
=\square+\ov{\square}+
[\partial,\db^*]+[\db,\partial^*].
\end{equation}
As $\square$, $\ov{\square}$ preserve the bigrading on 
$\Omega^{\bullet}(X)$,
and $[\partial,\db^*]: \Omega^{\bullet,\bullet}(X)
\to \Omega^{\bullet+1,\bullet-1}(X)$, 
we know that $\Delta$ preserves the bigrading on
$\Omega^{\bullet}(X)$ if and only if
\begin{equation}\label{lm2.65a}
[\partial,\db^*]=0.
\end{equation}

By the generalized K\"{a}hler identities \cite[(1.4.38d)]{MM07}
(cf. \cite{Demally86}) for $E=\C$ therein, we get
\begin{align}\label{herm19b}
\big[\Lambda,\partial\big]=
&\sqrt{-1}\big(\, \db^{*}+ \overline{\mathcal{T}}^*\, \big).
\end{align}
From \eqref{herm19b}, we get
\begin{equation}\label{herm20}
\Big[\partial,\db^*\Big]
=- \sqrt{-1}\Big[\partial,\big[\Lambda,\partial\big]\Big]
+ \Big[\partial,\overline{\mathcal{T}}^*\Big].
\end{equation}
But by  \eqref{jacobi}, we get
\begin{equation}\label{herm21}
\Big[\partial,\big[\Lambda,\partial\big]\Big]
=\Big[\Lambda,\big[\partial,\partial\big]\Big]+ 
\Big[\partial,\big[\partial, \Lambda\big]\Big].
\end{equation}
As $\big[\partial,\partial\big]=2 \partial^2=0$
and $\big[\partial, \Lambda\big]=- \big[\Lambda,\partial\big]$,
we get from  \eqref{herm21} that
\begin{equation}\label{herm22}
\Big[\partial,\big[\Lambda,\partial\big]\Big]=0.
\end{equation}
From \eqref{herm20} and \eqref{herm22}, we know that 
 \eqref{lm2.65a} is equivalent to 
\begin{equation}\label{herm23}
\big[\, \db^*, \mathcal{T}\big]=0.
\end{equation}
By \cite[(1.4.9)]{MM07},  the operator $\db^*$ has the form 
 $\db^* = -\sum_{j}i_{\ov{w}_j}\wi{\nabla}^{TX}_{w_j} 
 + 0\text{-order terms}$,
 here $\wi{\nabla}^{TX}$ is certain connection on 
 $\Lambda (T^*X)$, thus $\big[\, \db^*, \mathcal{T}\big]$
is a first order differential operator, and its principal symbol 
$\sigma$ is: for $\xi\in T^*X$, 
\begin{align}\label{herm25}
\sigma(\xi) = -\sqrt{-1}\sum_{j}(\xi, w_i)\cdot
\big[i_{\ov{w}_i}, \mathcal{T}\big].
\end{align}

By \cite[Lemma 1.4.10]{MM07}, 
\begin{equation}\label{herm18,1}
\mathcal{T}=-\frac{\sqrt{-1}}{2}\, 
\sum_{jkl} (\partial\Theta) (w_j,w_k, \ov{w}_l)
\Big[2\,w^k\wedge \ov{w}^l\wedge i_{\ov{w}_j}-2\,\delta_{jl} w^k
-w^j\wedge w^k\wedge i_{w_l}\Big]\,.
\end{equation}
From \eqref{herm18,1}, we get
\begin{multline}\label{herm27}
\big[i_{\ov{w}_i}, \mathcal{T}\big] 
= \sqrt{-1}\,\sum_{jkl} (\partial\Theta) (w_j,w_k, \ov{w}_l)
w^k\wedge \big[i_{\ov{w}_i},\ov{w}^l\big]\wedge i_{\ov{w}_j}\\
= \sqrt{-1}\,\sum_{jk} (\partial\Theta) (w_j,w_k, \ov{w}_i)w^k
\wedge i_{\ov{w}_j}.
\end{multline}
By \eqref{herm25} and \eqref{herm27},
the equation \eqref{herm23}
implies that 
\begin{align}\label{herm28}
\partial\Theta=0. 
\end{align}
Thus $\db\Theta=\ov{\partial\Theta}=0$ and $d \Theta=0$.
This means that 
if $\Delta$ preserves the bigrading on $\Omega^\bullet(X)$,
then $(X,J,\Theta)$ is K\"{a}hler.
\end{proof}

\begin{rem}\label{Weitzt1}
After we sent our preliminary version to 
Professor Bryant, he sent us an easier proof which works 
also in the almost-complex case. Here is the argument:

Let $(X,J, \Theta)$ be an almost complex manifold 
with almost complex structure $J$ and $\Theta$
a real $(1,1)$-form as in  \eqref{lm02.12}.
We suppose that the Hodge Laplacian $\Delta$ preserves 
the bigrading. 
In fact, we may only suppose that $\Delta$ sends $(0,1)$-forms 
to $(0,1)$-forms. In particular, $\Delta$ commutes with 
$J : TX \rightarrow TX$. Using the following lemma, 
this implies that $J$ is parallel with respect to 
the Levi-Civita connection $\nabla^{TX}$ on $(TX, g^{TX})$.
It is well-known (cf. \cite{MM07}) that this condition is equivalent 
to the metric being K\"{a}hler.

\begin{lemma}\label{Weitzt2}
Let $(X,g^{TX})$ be a 
Riemannian manifold 
and $L\in \cC^\infty(X,$ $\End(T^*X))$.
If $L$ commutes with the Hodge Laplacian $\Delta$
on $1$-forms, then $L$ is parallel  with respect to 
the Levi-Civita connection $\nabla^{TX}$.
\end{lemma}
\begin{proof}
Let $\nabla^{\Lambda (T^*X)}$ (resp. $\nabla^{\End (T^*X)}$)
be the connection on $\Lambda (T^\ast X)$ (resp. $\End (T^*X) $) 
induced by the Levi-Civita connection 
$\nabla^{TX}$ on $(TX, g^{TX})$. 

For $\nabla^F$ a connection on a vector bundle $F$, 
let $\Delta^{F}$ be
the Bochner Laplacian on $F$ associated to $\nabla^{F}$. 
By Definition, 
for  $\{e_{j}\}_{j}$ an orthonormal frame of $(TX,g^{TX})$, we have
\begin{align}\label{Weitz1.1}
\Delta^{F}
= -\sum_{j} \left[\left(\nabla^{F}_{e_{j}}\right)^2
-\nabla^{F}_{\nabla^{TX}_{e_{j}}e_{j}}\right].
\end{align}
\comment{Let $\Delta^{\Lambda (T^*X)}$ be the Bochner Laplacian 
on $\Lambda (T^\ast X)$ associated to 
$\nabla^{\Lambda (T^*X)}$. By Definition, 
for  $\{e_{j}\}_{j}$ an orthonormal frame of $(TX,g^{TX})$, we have
\begin{align}\label{Weitz1.2}
\Delta^{\Lambda (T^*X)}
= -\sum_{j} \left[\left(\nabla^{\Lambda (T^*X)}_{e_{j}}\right)^2
-\nabla^{\Lambda (T^*X)}_{\nabla^{TX}_{e_{j}}e_{j}}\right].
\end{align}
}

As $\nabla^{\Lambda (T^*X)}$  preserves  the $\Z$-grading 
on $\Lambda (T^*X)$, we know the Bochner Laplacian 
$\Delta^{\Lambda (T^*X)}$ on $\Lambda (T^\ast X)$ 
associated to $\nabla^{\Lambda (T^*X)}$,
also preserves the $\Z$-grading 
on $\Omega(X)$.
One can relate $\Delta$ and $\Delta^{\Lambda (T^*X)}$ 
by the Weitzenb\"{o}ck formula 
(cf. \cite[\S 3.6]{BeGeVe}). In particular, 
if $\alpha\in \Omega^1(X)$,
one has the equality
\begin{align}\label{Weitz}
\Delta \alpha =\Delta^{T^*X} \alpha + {\rm Ric} \,\alpha,
\end{align}
where the Ricci curvature Ric  is identified with a section
of the bundle $\End(T^*X)$ by means of $g^{TX}$.

By  (\ref{Weitz1.1}) and  (\ref{Weitz}), 
the principal symbol $\sigma_{2}(\Delta)$
of $\Delta$ is 
$\sigma_{2}(\Delta)(\xi)=|\xi|^2 \Id_{\Lambda (T^*X)}$
for $\xi\in T^*X$. Thus 
$\sigma_{2}(\Delta L - L \Delta)=0$ and 
$\Delta L - L \Delta$ is a first order differential operator.
 We now compute the principal symbol $\sigma_{1}(\Delta L 
 - L \Delta)$ by computing 
    $\lim_{t\to \infty} t^{-1}e^{-i t f} (\Delta L - L \Delta)e^{i t f}$ 
when $t\to + \infty$ for any $f\in \cC^\infty(X)$.
By   (\ref{Weitz1.1}) and  (\ref{Weitz}), we know
for any  $s\in\cC^\infty(X, T^*X)$
  \begin{align} \label{Weitz1.5}
\sigma_{1}(\Delta L - L \Delta)(df) s
=  \lim_{t\to \infty} t^{-1}e^{-i t f} (\Delta L - L \Delta)e^{i t f}  s
= -2i (\nabla^{\End (T^*X)}_{e_j}L) e_j(f) s.
  \end{align}

By assumption, one has 
\begin{align}\label{Weitz1.4}
\Delta L - L \Delta= 0.
\end{align}
This implies $\sigma_{1}(\Delta L - L \Delta)=0$.
Thus from (\ref{Weitz1.5}), we know (\ref{Weitz1.4})
implies 
 \begin{align} \label{Weitz1.6}
\nabla^{\End (T^*X)}L =0.
  \end{align}
The proof of Lemma \ref{Weitzt2} is completed.
\end{proof}

\end{rem}

\begin{rem}
In the first proof of Theorem \ref{almt2.3}, we use 
the generalized K\"{a}hler identity. When we began to clarify the 
situation of question \ref{qtncond}, when there is a distribution, 
we computed an analogue of the generalized K\"{a}hler identity
in this case. This is the object of the following subsection, 
which is independant from Subsection \ref{s2.3}.
\end{rem}

\subsection{A generalized sub-K\"{a}hler identity}\label{s2.2}

The Chern connection ${\nabla}^{T_hX}$ on $T_h X$ induces 
a connection on 
$TX$ and on the bundle 
$\Lambda^{\bullet,\bullet}(T^\ast X)$ (\cite[\S 1.2.2]{MM07}). 
This connection is denoted by $\widetilde{\nabla}^{TX}$.
In what follows, we identify $T_hX$ with $T^{(1,0)}X$ and thus 
we denote 
by $\nabla^{T^{(1,0)}X}$ the connection ${\nabla}^{T_hX}$ 
on $T^{(1,0)}X$.
For $v\in \cC^\infty(X, T^{(0,1)}X)$, we define 
$\nabla^{T^{(0,1)}X} v =\ov{\nabla^{T^{(1,0)}X}\ov{v}}$. Then 
$\widetilde{\nabla}^{TX} = \nabla^{T^{(1,0)}X}
\oplus \nabla^{T^{(0,1)}X}$.
 Moreover, we denote by $T \in \Lambda^2(T^\ast X) \otimes TX$ 
the torsion of $\widetilde{\nabla}^{TX}$.

By identifying $N$ in (\ref{cc1.12}) to $T_h X/W$, $N$ induces 
a holomorphic structure 
from $T_h X/W$. Let $\pi_{N}$ be the orthogonal projection
from $T_h X$ onto $N$. 
We denote by $\left\langle  ,\right\rangle$ the $\C$-bilinear form on 
$TX\otimes_{\R}\C$ induced by $g^{TX}$.

Let $h^{W}$, $h^{N}$ be the Hermitian metrics 
on $W, N$ induced by $h^{T_{h}X}$.
Let  ${\nabla}^{W}$, ${\nabla}^{N}$
be the Chern connections on $(W, h^{W})$, $(N, h^{N})$.
Then we have
\begin{align}\label{cc1.20}
    {\nabla}^{W}=\pi_W \nabla^{T_{h}X} \pi_W,
\quad {\nabla}^{N}=\pi_{N}\nabla^{T_{h}X} \pi_{N}.
\end{align}
As $W$ is a holomorphic subbundle, we know 
\begin{align}\label{cc1.21}
A= \nabla^{T_{h}X''}-  
({\nabla}^{W''}\oplus {\nabla}^{N \,''})
\in T^{*(0,1)}X\otimes \Hom (N,W).
\end{align}
The adjoint $A^{*}$ of $A$ takes values in
 $T^{*(1,0)}X\otimes\Hom (W,N)$. 
Note that for $w\in W, v\in N, U\in TX\otimes_\R \C$, we have
\begin{align}\label{cc1.22a}
    \left\langle  A^*(U)w,\ov{v} \right\rangle =
\left\langle  w, \ov{A(\ov{U})(v)}\right\rangle.
\end{align}
  Then, under the decomposition 
$T_{h}X= W\oplus N$, we have
\begin{align}\label{cc1.22}
\nabla^{T_{h}X} = \begin{pmatrix}  {\nabla}^{W} & A\\ 
	-A^{*}& {\nabla}^{N} \end{pmatrix}.
\end{align}

Let $\wi{\nabla}^W$, $\wi{\nabla}^N$ be the connection on 
$W_{\R}$, $N_{\R}$  induced by $\nabla^{W}$, $\nabla^{N}$ 
as above or as in \cite[(1.2.35)]{MM07}. Set 
\begin{align}\label{cc1.23}
{^\oplus\nabla}^{TX}= \wi{\nabla}^W\oplus \wi{\nabla}^N.
\end{align}
Let $\wi{\nabla}^W$, $^\oplus\wi{\nabla}^{TX}$ be the 
connections on  $\Lambda(W_\R^*)$, $\Lambda(T^*X)$
induced by $\nabla^W$, $^\oplus\nabla^{TX}$ 
as in \cite[\S 1.2.2]{MM07}, respectively.

Let $\{w_{j}\}_{j=1}^m$ be an orthonormal frame of $T^{(1,0)}X$ 
such that $\{w_{j}\}_{j=1}^{n}$ is an orthonormal frame of $W$.
Then by \cite[Lemma 1.4.4]{MM07}, we have
\begin{align}\begin{split}\label{lm2.20c}
\partial&=\sum_{j=1}^{m}w^j\wedge\widetilde{\nabla}^{TX}_{w_j}
+\frac{1}{2} \sum_{j,k,l=1}^{m}\langle T(w_j,w_k),
\ov{w}_l\rangle w^j\wedge w^k\wedge i_{w_l}\, ,
\end{split}
\end{align}
and 
\begin{align}\label{cc1.26}\begin{split}
\overline{\partial}^{*}   &= - \sum_{j=1}^{m} 
 i_{\ov{w}_j}\wi{\nabla}^{TX}_{w_j}
- \sum_{j,k=1}^{m}\langle T(w_j,w_k),\ov{w}_k\rangle  i_{\ov{w}_j}\\
 &\hspace{5mm}+\frac{1}{2}  \sum_{j,k,l=1}^{m}\langle 
 T(w_j,w_k),\ov{w}_l\rangle
\ov{w}^l\wedge i_{\ov{w}_k}\wedge i_{\ov{w}_j} .
\end{split}\end{align}
Note that for any $1\leq j,k\leq m$,  $U\in TX$, we have 
\begin{align}\label{cc1.26a}\begin{split}
&   (\wi{\nabla}^{TX}_U w^k, w_j) = -(w^k,\wi{\nabla}^{TX}_U w_j)
= -\langle\ov{w}_k,\wi{\nabla}^{TX}_U w_j\rangle 
= \langle\wi{\nabla}^{TX}_U\ov{w}_k, w_j\rangle,\\
&(\wi{\nabla}^{TX}_U \ov{w}^k, \ov{w}_j) 
= \langle\wi{\nabla}^{TX}_U w_k, \ov{w}_j\rangle.
\end{split}\end{align}
From (\ref{cc1.22}), for $1\leq k\leq m$, $n+1\leq \gamma\leq m$,
we get
\begin{align}\label{cc1.27a}\begin{split}
    \wi{\nabla}^{TX}_{w_k} \ov{w}_{\gamma}= &
    {^\oplus\nabla}^{TX}_{w_k} \ov{w}_{\gamma}
    + \sum_{j=1}^{n} \left\langle  \ov{A(\ov{w}_{k})w_{\gamma}},
    w_{j} \right\rangle  \ov{w}_{j}.
    \end{split}\end{align}
From (\ref{cc1.22}), (\ref{cc1.26a}) and (\ref{cc1.27a}),   
we know that on $\Lambda^{\bullet,\bullet}(T^*X)$,
\begin{align}\label{cc1.27}\begin{split}
\wi{\nabla}^{TX}_{w_k} 
=& {^{\oplus}\wi{\nabla}}^{TX}_{w_k} \\
&+  \sum_{\beta=n+1}^{m}\sum_{j=1}^{n} \Big[- \left\langle  
A^{*}(w_{k})w_{j}, \ov{w}_{\beta}\right\rangle   
\ov{w}^{\beta}\wedge i_{\ov{w}_{j}}
+ \left\langle  \ov{A(\ov{w}_{k})w_{\beta}},
    w_{j} \right\rangle  w^{j}\wedge i_{w_{\beta}}\Big]\\
=& {^{\oplus}\wi{\nabla}}^{TX}_{w_k} 
+  \sum_{\beta=n+1}^{m}\sum_{j=1}^{n} 
\left\langle \ov{A(\ov{w}_{k})w_{\beta}}, w_{j} \right\rangle  
\Big( - \ov{w}^{\beta}\wedge i_{\ov{w}_{j}}
+ w^{j}\wedge i_{w_{\beta}}\Big).
\end{split}\end{align}
By (\ref{cc1.17}),  (\ref{cc1.23}),  (\ref{lm2.20c}) 
and (\ref{cc1.27}), we know that
\begin{align}\label{cc1.25}\begin{split}
\partial_{W}=  \sum_{j=1}^{n}w^{j}\wedge
\left(\widetilde{\nabla}^{W}_{w_j}
+\frac{1}{2}\sum_{k,l=1}^{n}\langle T(w_j,w_k),
\ov{w}_l\rangle w^k\wedge i_{w_l}  \right).
\end{split}\end{align}

From (\ref{cc1.27}), we get for $n+1\leq \alpha\leq m$,
\begin{align}\label{cc1.28}
\pi_{W}
i_{\ov{w}_\alpha}\wi{\nabla}^{TX}_{w_\alpha} \pi_{W}
= -  \sum_{j=1}^{n}\left\langle  w_{j}, 
\ov{A(\ov{w}_{\alpha})w_{\alpha}}\right\rangle i_{\ov{w}_{j}}   .
\end{align}
From (\ref{cc1.17}), (\ref{cc1.26}),  (\ref{cc1.27}) 
and  (\ref{cc1.28}), we get
\begin{align}\label{cc1.29}\begin{split}
    \bar{\d}_{W}^{*} &=\sum_{j=1}^{n} 
\Big(-  i_{\ov{w}_j}\wi{\nabla}^{W}_{w_j}
- \sum_{k=1}^{m}\langle T(w_j,w_k),\ov{w}_k\rangle 
i_{\ov{w}_j}\Big)\\
 &\hspace{5mm}+\frac{1}{2} \sum_{j,k,l=1}^{n}
 \langle T(w_j,w_k),\ov{w}_l\rangle
\ov{w}^l\wedge i_{\ov{w}_k}\wedge i_{\ov{w}_j}
+\sum_{\alpha=n+1}^{m} \sum_{j=1}^{n}\left\langle  
w_{j}, \ov{A(\ov{w}_{\alpha})w_{\alpha}}\right\rangle 
i_{\ov{w}_{j}} \\
&=\sum_{j=1}^{n}  i_{\ov{w}_j}\Big\{- \wi{\nabla}^{W}_{w_j}
    +\frac{1}{2} \sum_{k,l=1}^{n}\langle T(w_j,w_k),\ov{w}_l\rangle
    i_{\ov{w}_k}  \wedge \ov{w}^l\\
  &\hspace{20mm}  + \sum_{\alpha=n+1}^{m} \Big(\left\langle  
    w_{j}, \ov{A(\ov{w}_{\alpha})w_{\alpha}}\right\rangle
- \langle T(w_j,w_\alpha),\ov{w}_\alpha\rangle\Big) \Big\}    .
\end{split}\end{align}
We also generalize the definition of the operators
$L = \Theta \wedge$ 
and $\Lambda$ its adjoint, by defining 
$\Theta_W \in \Lambda^{1,1} (W_{\R}^*)$ as the restriction 
to $\Lambda^{1,1} (W_{\R}^*)$ of $\Theta$. 
We thus get operators $L_W$ and $\Lambda_W$ on 
$\Lambda^{\bullet,\bullet}(W_{\R}^*)$.  By \eqref{herm15,1}, we have
\begin{align} \label{herm15a}\begin{split}
&L_{W}=\sqrt{-1}  \sum_{j=1}^{n} w^j\wedge \ov{w}^j\wedge\,,
\quad \Lambda_{W}=- \sqrt{-1}  \sum_{j=1}^{n} 
i_{\ov{w}_j}i_{w_j}\, .\\
\end{split}\end{align}
The \emph{Hermitian torsion operator} 
is defined as in \eqref{herm18} and \cite[(1.4.34)]{MM07} by
\begin{align} \label{cc1.30}
\mT_W := [\Lambda_W, \d_W\Theta_W].
\end{align}

We have the analogue of \cite[Theorem 1.4.11]{MM07}.
\begin{prop}{Generalized sub-K\"{a}hler identity}
    \begin{multline}\label{subKa}
[\Lambda_W,\d_W] =  \sqrt{-1}
\left(\bar{\d}_W^\ast + \overline{\mT}_W^\ast\right)\\
 - \sqrt{-1} \sum_{\alpha=n+1}^{m}\sum_{j=1}^{n}
\Big(\left\langle w_{j}, \ov{A(\ov{w}_{\alpha})w_{\alpha}}\right\rangle
- \langle T(w_j,w_\alpha),\ov{w}_\alpha\rangle\Big)  
i_{\ov{w}_{j}}  \,\pi_{W} .
    \end{multline}	
\end{prop}
\begin{proof} Set $\pi_{W,\bot}= \Id - \pi_{W}$.
    By (\ref{herm19b}), we have 
    \begin{align}\label{cc1.31}
 \pi_{W}   \big[\Lambda,\partial\big]  \pi_{W}   =
    &\sqrt{-1}\big(\, \db^{*}_{W} 
    + \pi_{W}  \overline{\mathcal{T}}^*\pi_{W}   \, \big).
    \end{align}
    Note that 
    \begin{align}\label{cc1.32}
\pi_{W} \Lambda  \pi_{W} = \Lambda  \pi_{W}  .
\end{align}
From (\ref{cc1.32}), we know
\begin{multline}\label{cc1.34}
    \pi_{W}   \big[\Lambda,\partial\big]  \pi_{W} 
    =   \pi_{W}   \Lambda  \pi_{W} \partial  \pi_{W} 
    +   \pi_{W}   \Lambda  \pi_{W, \bot} \partial  \pi_{W} 
    -   \pi_{W}  \partial \pi_{W}   \Lambda  \pi_{W} \\
    = \big[\Lambda_{W},\partial_{W}\big] 
    +  \pi_{W}   \Lambda  \pi_{W, \bot} \partial  \pi_{W}.
\end{multline}
By (\ref{lm2.20c}) and (\ref{cc1.27}), we have
\begin{multline}\label{cc1.35}
    \pi_{W}   \Lambda  \pi_{W, \bot} \partial  \pi_{W}
    = - \sqrt{-1}  \pi_{W} 
    \sum_{\gamma=n+1}^{m}i_{\ov{w}_{\gamma}}i_{w_{\gamma}}
    \pi_{W, \bot} \partial  \pi_{W}\\
    = \sqrt{-1}  \sum_{\alpha=n+1}^{m}\sum_{j=1}^{n}
    \left\langle w_{j}, \ov{A(\ov{w}_{\alpha})w_{\alpha}}
    \right\rangle i_{\ov{w}_{j}} \pi_{W} \, . 
\end{multline}

By (\ref{herm15,1}), (\ref{lm2.20c}) and (\ref{cc1.27}), we have 
$\partial_{W}\Theta_{W}= \pi_{W}\partial \Theta\pi_{W}$. 
Thus similarly to (\ref{cc1.34}), we have 
\begin{align}\label{cc1.36}
    \pi_{W} [\Lambda, \partial \Theta]\pi_{W} =
\left[\Lambda_{W}, \partial_{W} \Theta_{W}\right]
+  \pi_{W}   \Lambda  \pi_{W, \bot} \partial \Theta \pi_{W}.
\end{align}
By \cite[(1.2.48), (1.2.54)]{MM07}, we have
\begin{align}\label{cc1.37}
    \partial\Theta =  \frac{\sqrt{-1}}{2} \sum_{i,j,k=1}^{m}
\langle T(w_i,w_j),\ov{w}_k\rangle w^i\wedge w^j
\wedge \ov{w}^k.
\end{align}
From (\ref{herm15,1}) and (\ref{cc1.37}), we know 
\begin{align}\label{cc1.38}
  \pi_{W}   \Lambda  \pi_{W, \bot} \partial \Theta \pi_{W}
= -\sum_{\alpha=n+1}^{m}\sum_{j=1}^{n}
\left\langle  T(w_{\alpha}, w_{j}), 
\ov{w}_{\alpha}\right\rangle w^{j} \pi_{W}    .
\end{align}
Taking the adjoint of  (\ref{cc1.36}), from 
(\ref{cc1.30}) and (\ref{cc1.38}),
we know 
\begin{align}\label{cc1.39}
    \pi_{W} \mathcal{T}^*\pi_{W} = \mT_W^\ast
-  \sum_{\alpha=n+1}^{m} \sum_{j=1}^{n}
\left\langle  T(\ov{w}_{\alpha}, \ov{w}_{j}), 
    w_{\alpha}\right\rangle  i_{w_{j}} \pi_{W} .
\end{align}
Thus
\begin{align}\label{cc1.40}
    \pi_{W} \overline{\mathcal{T}}^*\pi_{W} 
    = \overline{\mT}_W^\ast
- \sum_{\alpha=n+1}^{m} \sum_{j=1}^{n}
\left\langle T({w}_{\alpha}, {w}_{j}), 
    \ov{w}_{\alpha}\right\rangle  i_{\ov{w}_{j}} \pi_{W} .
\end{align}
Finally from  (\ref{cc1.31}), (\ref{cc1.34}), (\ref{cc1.35})
and  (\ref{cc1.40}), we get
\begin{multline}\label{cc1.41}
    [\Lambda_W,\d_W] =  \sqrt{-1}
    \left(\bar{\d}_W^\ast + \overline{\mT}_W^\ast\right)\\
 -   \sqrt{-1}  \sum_{\alpha=n+1}^{m}\sum_{j=1}^{n}
 \Big( \left\langle w_{j}, \ov{A(\ov{w}_{\alpha})w_{\alpha}}
 \right\rangle
+ \left\langle  T({w}_{\alpha}, {w}_{j}), 
     \ov{w}_{\alpha}\right\rangle  \Big)i_{\ov{w}_{j}}\pi_{W} .
\end{multline}
   From  \eqref{cc1.41}, we get \eqref{subKa}.
\end{proof}

\subsection{The proof of Theorem \ref{thmprin}}\label{s2.3}

Remember the construction of the holomorphic map 
$\varphi : F\widehat{\otimes} \Lambda^\bullet W^\ast 
\rightarrow \Lambda^{\bullet+2} W^\ast$ in 
Subsection \ref{subscompl}. 
There we assumed for simplicity that this map has constant rank. 
For the purpose of Theorem \ref{thmprin}, one can easily reduce
to this case. 
Indeed, there exists an analytic subset $V$ of $X$ 
such that $\varphi: F\widehat{\otimes}\Lambda^k W^* 
\to \Lambda^{k+2} W^*$ has maximum rank on $X\setminus V$ 
for any $k$. In particular, 
$\varphi(F\widehat{\otimes}\Lambda^k W^*)$ forms 
a vector subbundle of $\Lambda^{k+2} W^*$ on $X\setminus V$
for any $k$. We can define the vector bundles $F_\varphi$
and $F_{\varphi,\bot}$ on $X \setminus V$ as before. 
On $X \setminus V$, we have the decompositions (\ref{cc2.3b}) 
and (\ref{cc2.4}) for forms with compact support; in particular

\begin{align}\label{cc2.4b}\begin{split}
\mQ\cap \Omega^\bullet_{c}(X\setminus V) 
= \cC^\infty_{c} (X\setminus V, 
F_{\varphi,\bot}\widehat{\otimes} \ov{F}_{\varphi,\bot}).
    \end{split}\end{align}

As $X\setminus V$ is an open connected dense subset of $X$, 
the characteristic Laplacian $\Delta_\mQ$ 
preserves the bigrading on $\mQ$ if and only if it preserves 
the bigrading on $\mQ\cap \Omega^\bullet_{c}(X\setminus V)$. 
Thus we can work on $X \setminus V$ instead of $X$.
   
 From the above discussion, in the rest, we will assume that 
 $\varphi(F\widehat{\otimes}\Lambda^k W^*)$ forms a vector 
 subbundle of $\Lambda^{k+2} W^*$ on $X$ for any $k$.
Then we can use the formalism developed in 
Subsection \ref{subscompl}.

By (\ref{cc1.3}), (\ref{cc1.4}), as in (\ref{lm2.65}),  we have 
\begin{equation}\label{cc2.6}
\Delta_\mQ=\Big[\partial_\mQ+\db_\mQ,\partial^*_\mQ
+\db^*_\mQ\Big]
=\square_\mQ+\ov{\square}_\mQ
+ \left[\partial_\mQ,\db^*_\mQ\right]
+ \Big[\, \db_\mQ,\partial^*_\mQ\Big].
\end{equation}
As $\square_\mQ$, $\ov{\square}_\mQ$ preserve the bigrading 
on $\mQ$, and $[\partial_\mQ,\db^*_\mQ]: \mQ^{\bullet,\bullet}
\to \mQ^{\bullet+1,\bullet-1}$, 
we know that $\Delta_\mQ$ preserves the bigrading on $\mQ$
if and only if
\begin{equation}\label{cc2.7}
    \left[\partial_\mQ,\db^*_\mQ\right]=0.
\end{equation}

We would like to understand the 
operator $[\partial_\mQ,\db^*_\mQ]$.

For $f\in \cC^\infty(X)$,  by (\ref{cc1.17}), we have
\begin{align}\label{cc2.8}
   \partial_{W}f = \sum_{j=1}^n w_{j}(f) w^j\in W^*,
    \end{align}
where $\{w_{j}\}_{j=1}^m$ is an orthonormal frame of $T^{(1,0)}X$ 
such that $\{w_{j}\}_{j=1}^{n}$ is an orthonormal frame of $W$.
For $\xi\in T^*_{\R}X$, let $\xi^*\in T_{\R}X$ be the metric 
dual of $\xi$. In particular, if $\xi \in W^*$, then $\xi^*\in \ov{W}$.

Since $\mJ$ is stable by $d,\d,\bar{\d}$, as in  (\ref{cc1.5}), 
we have as maps on $\Omega(X)$,
\begin{align} \label{cc2.9}\begin{split}
&\pi_\mQ \circ \partial\circ \pi_\mQ = \pi_\mQ \circ \partial,\\
&\pi_\mQ \circ \bar{\d}\circ \pi_\mQ = \pi_\mQ \circ \bar{\d},
\quad \pi_\mQ\circ \bar{\d}^{*}\circ \pi_\mQ=\bar{\d}^{*}
\circ \pi_\mQ  .
\end{split}\end{align}

Let $h^{F_\varphi}, h^{F_{\varphi,\bot}}$ be the Hermitian metrics 
on $F_\varphi, F_{\varphi,\bot}$ induced by $h^{\Lambda W^*}$
on $\Lambda W^*$, which is induced by $h^W$. We recall that 
$F_\varphi $ is a holomorphic vector subbundle of $\Lambda W^*$. 
As in \eqref{cc1.20}, 
let $\nabla^{F_{\varphi}}$, $\nabla^{F_{\varphi,\bot}}$ be 
the Chern connections on 
$(F_\varphi, h^{F_{\varphi}})$, $(F_{\varphi,\bot}, 
h^{F_{\varphi,\bot}})$. 
Let  $\nabla^{\Lambda W^*}$ be the connection on 
$\Lambda W^*$ induced by 
$\nabla ^W$, then $\nabla^{\Lambda W^*}$ is the 
Chern connection on 
$(\Lambda W^*,  h^{\Lambda W^*})$. Set 
\begin{align}\label{cc2.11}
B= \nabla^{\Lambda W^*\,''}-  
({\nabla}^{F_{\varphi}\, ''}\oplus {\nabla}^{ F_{\varphi,\bot}\,''})
\in T^{*(0,1)}X\otimes \Hom (F_{\varphi,\bot},F_{\varphi}).
\end{align}
The adjoint $B^{*}$ of $B$ takes values in
 $T^{*(1,0)}X\otimes\Hom (F_{\varphi},F_{\varphi,\bot})$. 
  Then under the decomposition 
$\Lambda W^*= F_{\varphi}\oplus F_{\varphi,\bot}$, we have
\begin{align}\label{cc2.12}
\nabla^{\Lambda W^*}= \begin{pmatrix}{\nabla}^{F_{\varphi}} & B\\ 
	-B^{*}& {\nabla}^{F_{\varphi,\bot}} \end{pmatrix}.
\end{align}
We denote also $\pi_{\mQ}$ the orthogonal projection from
$\Lambda W^*\widehat{\otimes}\Lambda\ov{W}^*$
onto $F_{\varphi,\bot}\widehat{\otimes}\ov{F}_{\varphi,\bot}$,
and $\pi_{\mQ}^\bot= \Id_{\Lambda W^*\widehat{\otimes}
\Lambda\ov{W}^*}-\pi_{\mQ}$. Then
\begin{align} \label{cc2.9a}\begin{split}
&\pi_\mQ =  \pi_{F_{\varphi,\bot}}
\otimes \ov{\pi}_{F_{\varphi,\bot}},\\
&\pi_{\mQ}^\bot=  \pi_{F_{\varphi}}
\otimes \ov{\pi}_{F_{\varphi,\bot}}
+ \pi_{F_{\varphi,\bot}}\otimes \ov{\pi}_{F_{\varphi}}
+\pi_{F_{\varphi}}
\otimes \ov{\pi}_{F_{\varphi}} .
\end{split}\end{align}

\begin{lemma}\label{cct2.6}
The operator $[\partial_\mQ,\db^*_\mQ]$ is a first order 
differential operator acting on 
$F_{\varphi,\bot}\widehat{\otimes}\ov{F}_{\varphi,\bot}$.
 Its principal symbol, evaluated on $\xi \in T^\ast X$, is $\sqrt{-1}$ times 
\begin{multline}\label{cc2.14}
 \pi_{\mQ}\sum_{j,k=1}^{n}  i_{\ov{w}_j}  w^{k} 
 \wedge   \Big[ (\xi_W,  T(w_{j},w_{k})) 
  - (\xi, \pi_{N}[w_{j},w_{k}]) \Big]\pi_{\mQ}\\
+  \sum_{j=1}^{n}  (B^*(w_j) \pi_{F_{\varphi}}\xi_W ) 
\widehat{\otimes} i_{\ov{w}_j} 
-  \pi_{\mQ} \sum_{j=1}^{n} 
w^j \widehat{\otimes} i_{(\xi_W)^*}  \ov{B(\ov{w_j})},
\end{multline} 
where $\xi_W$ is the orthogonal projection of $\xi$ on $W^\ast$
\end{lemma}

\begin{rem}\label{cct2.8}
Theorem \ref{thmprin} is an easy corollary of (\ref{cc2.14}).
Indeed, if we take $\xi$ a holomorphic one-form which
is orthogonal to $W^\ast$, the principal symbol is 
\begin{align*}
-\sqrt{-1}  \pi_{\mQ}\sum_{j,k=1}^{n}  
i_{\ov{w}_j}  w^{k} \wedge (\xi, \pi_{N}[w_{j},w_{k}]) 
\pi_{\mQ}.
\end{align*}

Evaluating at $\ov{w}^j$, which belongs to $F_{\varphi,\bot}$, 
one gets
\begin{align*}
\sqrt{-1} \sum_{k=1}^n (\xi, \pi_{N}[w_{j},w_{k}]) w^{k}.
\end{align*}
This term vanishes for any $j$ and $\xi$ if and only 
if the distribution is involutive, which shows Theorem \ref{thmprin}.
\end{rem}

\begin{proof}[Proof of Lemma \ref{cct2.6}]  
 Note that for $\xi\in W^{*}$, 
 $\psi\in F_{\varphi,\bot}\widehat{\otimes} \ov{F}_{\varphi,\bot}$,
 by (\ref{cc2.3}), (\ref{cc2.10}) and (\ref{cc2.9a}),  we know that 
 $ \pi_{\mQ}^{\bot} (\xi\wedge \psi)\in 
 F_{\varphi}\widehat{\otimes} \ov{F}_{\varphi,\bot}$.
 Thus by  (\ref{cc2.10}), $ i_{\xi^{*}}\pi_{\mQ}^{\bot} (\xi\wedge \psi)\in 
 F_{\varphi}\widehat{\otimes} \ov{F}_{\varphi,\bot}$,
as $\xi^{*}\in \ov{W}$,  and this implies
 \begin{align}\label{cc2.18}
     \pi_{\mQ} i_{\xi^*}  \pi_{\mQ}^{\bot} \xi\wedge 
     \pi_{\mQ} =0 
\quad \text{  for any  }  \xi\in W^{*}	  .
\end{align}

 We compute now the principal symbol of 
 $[\partial_\mQ,\db^*_\mQ]$
 by computing the asymptotics of
 $e^{-i t f} [\partial_\mQ,\db^*_\mQ]e^{i t f}$ when 
 $t\to + \infty$ for $f\in \cC^\infty(X)$.
 By (\ref{cc1.18a}),  (\ref{cc1.25}) and (\ref{cc1.29}), we have first
    \begin{align} \label{cc2.19}
    \begin{split}
  e^{-i t f} \partial_\mQ e^{i t f}  =   i t\, 
\pi_{\mQ}  \partial_{W}f \wedge \pi_{\mQ}   
  + \partial_\mQ,\\
 e^{-i t f} \db^*_\mQ e^{i t f} = - i t \, 
\pi_{\mQ}i_{(\partial_{W}f )^*}  \pi_{\mQ}   
+  \db^*_\mQ.
    \end{split}
    \end{align}

Thus from (\ref{cc2.10}),  (\ref{cc2.19}), the principal symbol of 
$[\partial_\mQ,\db^*_\mQ]$ as a second order differential 
operator is $\lim_{t\to\infty} t^{-2} e^{-i t f} 
[\partial_\mQ,\db^*_\mQ]e^{i t f}$, that is 
    \begin{multline} \label{cc2.20}
  \Big[i  \pi_{\mQ}\partial_{W}f 
  \wedge \pi_{\mQ}\, , 
 - i  \pi_{\mQ}i_{(\partial_{W}f )^*}
 \pi_{\mQ}\Big]\\
  = \pi_{\mQ}\partial_{W}f \wedge  
  i_{(\partial_{W}f )^*}     \pi_{\mQ}
  + \pi_{\mQ} i_{(\partial_{W}f )^*}  
 \pi_{\mQ}\partial_{W}f \wedge  
  \pi_{\mQ}\\
 =-\pi_{\mQ} i_{(\partial_{W}f )^*}  
 \pi_{\mQ}^{\bot} \partial_{W}f \wedge \pi_{\mQ} =0,
    \end{multline}
 here we use  (\ref{cc2.18})  in the last equality.
 The equation  (\ref{cc2.20})  means that 
 $[\partial_\mQ,\db^*_\mQ]$ is a first order differential operator.
    
   By  (\ref{cc2.19}), the principal symbol of 
 $[\partial_\mQ,\db^*_\mQ]$ as a first order differential operator is
 $\lim_{t\to \infty} t^{-1} e^{-i t f} 
 [\partial_\mQ,\db^*_\mQ]e^{i t f}$, that is 
  \begin{align}\label{cc2.21}
i\left[\pi_{\mQ}\partial_{W}f \wedge  
\pi_{\mQ},    \db^*_\mQ\right]
-i\left[  \partial_\mQ, \pi_{\mQ}i_{(\partial_{W}f )^*}  
\pi_{\mQ} \right].
\end{align}

By (\ref{cc2.10}) and (\ref{cc1.18a}),  we get
\begin{multline}\label{cc2.22}
    \left[\pi_{\mQ}\partial_{W}f \wedge  
    \pi_{\mQ},    \db^*_\mQ\right]
    = \pi_{\mQ}\partial_{W}f \wedge  
    \bar{\d}_W^\ast \pi_{\mQ}
    +  \pi_{\mQ}\bar{\d}_W^\ast \pi_{\mQ}
    \partial_{W}f \wedge   \pi_{\mQ}\\
    =  \pi_{\mQ} \bar{\d}_W^\ast (\partial_{W}f) 
    \pi_{\mQ} - \pi_{\mQ}\bar{\d}_W^\ast 
    \pi_{\mQ}^{\bot}\partial_{W}f \wedge   \pi_{\mQ}.
\end{multline}
Again by (\ref{cc2.10}) and (\ref{cc1.18a}),  we get
\begin{multline}\label{cc2.24}
-    \left[   \partial_\mQ, \pi_{\mQ}i_{(\partial_{W}f )^*}  \wedge  
    \pi_{\mQ}\right]\\
    = -\pi_{\mQ}i_{(\partial_{W}f )^*} 
    \pi_{\mQ}\d_W\pi_{\mQ}
    -\pi_{\mQ}\d_W i_{(\partial_{W}f )^*} 
      \pi_{\mQ}\\
    =   - \pi_{\mQ}  \d_W( i_{(\partial_{W}f )^*} ) 
    \wedge \pi_{\mQ} 
 + \pi_{\mQ} i_{(\partial_{W}f )^*}  \pi_{\mQ}^{\bot}
    \d_W\pi_{\mQ}.
\end{multline}

By  (\ref{cc1.29}) and  (\ref{cc2.8}), we get
\begin{multline}\label{cc2.25}
    \bar{\d}_W^\ast (\partial_{W}f) 
 = [\bar{\d}_W^\ast,  \partial_{W}f]
  =- \left[\sum_{j=1}^{n}  i_{\ov{w}_j} \wi{\nabla}^{W}_{w_j},  
  \partial_{W}f\right]\\
  = -\sum_{j,k=1}^{n}  i_{\ov{w}_j}  \Big[w_{j}(w_{k}(f))  w^{k} 
  + w_{k}(f)   \wi{\nabla}^{W}_{w_j}w^{k} \Big]\\
  =-\sum_{j,k=1}^{n}  i_{\ov{w}_j}  w^{k} 
  \left[w_{j}(w_{k}(f))- (\partial_{W}f, \nabla^{W}_{w_j}w_{k})\right].
\end{multline}
By   (\ref{cc1.25}) and  (\ref{cc2.8}), we get
\begin{multline}\label{cc2.26}
  -  \d_W( i_{(\partial_{W}f )^*} )   
 = -[ \d_W,  i_{(\partial_{W}f )^*} ]
  =- \left[\sum_{k=1}^{n}  w^{k} \wi{\nabla}^{W}_{w_k},  
  i_{(\partial_{W}f )^*} \right]\\
  = -\sum_{j,k=1}^{n}  w^{k} \wedge 
\Big[  w_{k}(w_{j}(f))   i_{\ov{w}_j} 
  + w_{j}(f)   i_{\nabla^{W}_{w_k}\ov{w}_{j}}\ \Big]\\
  =\sum_{j,k=1}^{n}  i_{\ov{w}_j}  w^{k} 
  \Big[w_{k}(w_{j}(f))- (\partial_{W}f, \nabla^{W}_{w_k}w_{j})\Big].
\end{multline}
By  (\ref{cc1.22}),  (\ref{cc2.25}) and  (\ref{cc2.26}), we get
\begin{multline}\label{cc2.27}
 \bar{\d}_W^\ast (\partial_{W}f) -  \d_W( i_{(\partial_{W}f )^*})\\   
  =\sum_{j,k=1}^{n}  i_{\ov{w}_j}  w^{k} 
  \left[ (\partial f, - [w_{j},w_{k}]) 
  + (\partial_{W}f, 
  \nabla^{W}_{w_j}w_{k}-\nabla^{W}_{w_k}w_{j})\right]\\
  =\sum_{j,k=1}^{n}  i_{\ov{w}_j}  w^{k} 
  \Big[ (\partial_{W}f,  T(w_{j},w_{k})) 
  - (\partial f, \pi_{N}[w_{j},w_{k}]) \Big].
\end{multline}
From   (\ref{cc2.21})--(\ref{cc2.27}),  we know that the principal 
symbol of the first order differential operator 
$[\partial_\mQ,\db^*_\mQ]$ is $\sqrt{-1}$ times 
 \begin{multline}\label{cc2.28}
 \pi_{\mQ}\sum_{j,k=1}^{n}  i_{\ov{w}_j}  w^{k} 
  \left[ (\partial_{W}f,  T(w_{j},w_{k})) 
  - (\partial f, \pi_{N}[w_{j},w_{k}]) \right]\pi_{\mQ}\\
 - \pi_{\mQ}\bar{\d}_W^\ast 
    \pi_{\mQ}^{\bot}\partial_{W}f \wedge   \pi_{\mQ}
  + \pi_{\mQ} i_{(\partial_{W}f )^*}  \pi_{\mQ}^{\bot}
    \d_W\pi_{\mQ} .
\end{multline} 

Now by  \eqref{cc2.9a}, 
$ \pi_{\mQ}^{\bot}\partial_{W}f \wedge   
\pi_{\mQ}\subset
F_{\varphi}\otimes\ov{F}_{\varphi,\bot}$. By
\eqref{cc2.10},  \eqref{cc1.29} and  \eqref{cc2.12}, 
 we know that
\begin{multline}\label{cc2.32}
\pi_{\mQ} \bar{\d}_W^\ast \pi_{\mQ}^{\bot}
\partial_{W}f \wedge \pi_{\mQ} 
= \pi_{\mQ} \sum_{j=1}^{n} 
\Big(-  i_{\ov{w}_j}\wi{\nabla}^{W}_{w_j}\Big)
\pi_{\mQ}^{\bot}\partial_{W}f  \wedge\pi_{\mQ} \\
= \pi_{\mQ} \sum_{j=1}^{n} 
\Big(  i_{\ov{w}_j} B^*(w_j)\otimes 1\Big)\pi_{\mQ}^{\bot}
\partial_{W}f \wedge \pi_{\mQ} 
=- \sum_{j=1}^{n}  (B^*(w_j) \pi_{F_{\varphi}}
\partial_{W}f \wedge) \widehat{\otimes} i_{\ov{w}_j} .
\end{multline}
Let $P$ be the orthogonal projection from 
$\Lambda^{\bullet,\bullet}(W^{*}_{\R})$ onto
$F_{\varphi,\bot}\otimes \ov{F}_{\varphi}$.
Note that $ \pi_{\mQ}^{\bot}\partial_{W}  
\pi_{\mQ} \subset F_\varphi\otimes 
\Lambda \ov{W}^* \oplus F_{\varphi,\bot}\otimes 
\ov{F}_{\varphi}$,
as $ i_{(\partial_{W}f )^*} F_\varphi\otimes 
\Lambda \ov{W}^* \subset F_\varphi\otimes \Lambda \ov{W}^*$,  
from \eqref{cc1.25}, \eqref{cc2.12} and \eqref{cc2.9a},  we have also 
\begin{multline}\label{cc2.33}
\pi_{\mQ} i_{(\partial_{W}f )^*}  
\pi_{\mQ}^{\bot}\partial_{W}  \pi_{\mQ} 
=\pi_{\mQ} i_{(\partial_{W}f )^*}  
P \partial_{W}  \pi_{\mQ} \\
= \pi_{\mQ} i_{(\partial_{W}f )^*}  
P  \sum_{j=1}^{n} w^j\wedge \wi{\nabla}^{W}_{w_{j}}\pi_{\mQ} 
=-  \pi_{\mQ} \sum_{j=1}^{n} 
w^j \widehat{\otimes} i_{(\partial_{W}f )^*}  \ov{B(\ov{w_j})}.
\end{multline}
The proof of Lemma \ref{cct2.6} is completed.
\end{proof}


\def\cprime{$'$} \def\cprime{$'$}
\providecommand{\bysame}{\leavevmode\hbox to3em{\hrulefill}\thinspace}
\providecommand{\MR}{\relax\ifhmode\unskip\space\fi MR }
\providecommand{\MRhref}[2]{%
  \href{http://www.ams.org/mathscinet-getitem?mr=#1}{#2}
}
\providecommand{\href}[2]{#2}

\end{document}